\documentclass[11pt,a4paper]{article}

\usepackage{graphicx}
\usepackage{amsmath}
\usepackage{amsfonts}
\usepackage{amssymb}
\usepackage{placeins}
\usepackage{enumerate}
\usepackage{lscape,hyperref}
\usepackage{longtable}
\usepackage{microtype}
\usepackage{rotating}
\usepackage{subcaption}
\usepackage{color}
\usepackage{url}
\usepackage{amsthm}
\usepackage{float}
\usepackage{booktabs}  
\usepackage{afterpage} 
\usepackage{rotating}
\usepackage[title,titletoc,toc]{appendix}
\usepackage{mathtools}
\newtheorem{theorem}{Theorem}[section]

\usepackage[ruled,vlined,noline,linesnumbered]{algorithm2e}
\usepackage{algorithmic}
\usepackage[numbers,sort&compress]{natbib}
\usepackage{textcomp}
\usepackage[dvipsnames,svgnames]{xcolor}
\usepackage[normalem]{ulem}

\usepackage{tikz}

\usetikzlibrary{shapes,arrows,arrows.meta,calc}
\usepackage{tikz}
\usepackage{subcaption}
\usepackage{pifont}
\definecolor{okgreen}{RGB}{59,109,17}
\definecolor{badred}{RGB}{200,40,40}

\tikzset{
  omegaset/.style  = {fill=gray!15, draw=gray!70, line width=0.5pt},
  ring/.style      = {draw=okgreen, dashed, line width=0.6pt},
  ringfill/.style  = {fill=okgreen!45, even odd rule},
  badring/.style   = {draw=badred, dashed, line width=0.6pt, fill=badred!10, even odd rule},
  incumbent/.style = {circle, fill=black, inner sep=1.2pt},
}

\newtheorem{corollary}{Corollary}[section]

\newtheorem{definition}{Definition}[section]

\newtheorem{lemma}{Lemma}[section]

\newtheorem{proposition}{Proposition}[section]

\newtheorem{assumption}{Assumption}[section]

\SetCommentSty{mycommfont}
\def\real{\mathbb{R}}

\DeclareMathOperator{\glb}{global}
\DeclareMathOperator{\lcl}{local}

\newcommand{\email}[1]{Email: #1}

\title{STREGO: A Scalable Implementation of a Trust-Region Framework for Noisy Bayesian Optimization }

\author{
Y. Diouane
\thanks{GERAD \& Department of Mathematics and Industrial Engineering, Polytechnique Montr\'eal.
\email{{\tt youssef.diouane@polymtl.ca}}}
\and
I. El Hallaoui
\thanks{GERAD \& Department of Mathematics and Industrial Engineering, Polytechnique Montr\'eal.
\email{{\tt issmail.el-hallaoui@polymtl.ca}}}
\and
M. Kadri
\thanks{GERAD \& Department of Mathematics and Industrial Engineering, Polytechnique Montr\'eal.
\email{{\tt meisseme.kadri@etud.polymtl.ca}}}
\and
A. Lamghari
\thanks{GERAD \& Department of Management, Universit\'e du Qu\'ebec \`a Trois-Rivi\`eres.
\email{{\tt amina.lamghari@uqtr.ca}}}
}

\date{\today}

\begin{document}
\maketitle

\begin{abstract}
TREGO, a trust-region Bayesian optimization (BO) method for expensive blackbox problems, has convergence guarantees, but only for exact function values. In high dimension, its performance tends to deteriorate. This paper asks whether a trust-region BO method can be extended to handle noisy evaluations while being competitive in high dimension. We first generalize TREGO into a trust-region framework for noisy BO and extend its convergence analysis to noisy function estimates that are sufficiently accurate with sufficiently high probability. We then propose STREGO, a scalable implementation of this framework. Its global phase selects a batch of candidates by solving a bi-objective problem on the posterior mean of a Gaussian process and the integrated variance reduction. Unlike the predictive variance, this second objective does not favor boundary points merely because they are far from the data. Its local phase minimizes the posterior mean of a second Gaussian process, trained only on points close to the incumbent. STREGO is compared with TREGO, TURBO, BAXUS and NSMA-BO on synthetic and real-world problems of dimension $4$ to $124$, with and without noise. It gives the best average results in every synthetic setting tested and the best or tied-best solution on four of the six real-world problems. An implementation of STREGO is available at \url{https://github.com/Mayousk/STREGO}.
\end{abstract}

\noindent
{\small
\textbf{Funding:} This research is funded by the Natural Sciences and Engineering Research Council of Canada (NSERC) discovery grants  
 RGPIN-2024-05093 (Diouane),
    RGPIN-2024-06137 (El Hallaoui),
    and RGPIN-2019-07172 (Lamghari). 
}

\section{Introduction}
\label{sec:intro}
In this paper, we consider the blackbox optimization problem
\begin{equation}
  \min_{x \in \Omega} \; f(x).
  \label{eq:main-prob}
\end{equation}
The objective function $f \colon \Omega \subseteq \mathbb{R}^{d} \to \mathbb{R}$ is locally Lipschitz continuous and bounded from below but cannot be evaluated directly. Optimization algorithms have access only to a stochastic zeroth-order oracle $f_{\Theta}(x)$, which returns a noisy observation of $f(x)$ and depends on a random variable $\Theta$ with an unknown distribution. The feasible set $\Omega$ is a closed convex set with a nonempty interior, defined by unrelaxable constraints~\cite{LedWild2015}, typically bound constraints. We focus on instances in which each call to the oracle is expensive and the dimension $d$ is large, typically $d \geq 20$~\cite{moriconi2020high}.

Expensive calls limit the number of evaluations, and Bayesian optimization (BO)~\cite{mockus1989bayesian,shahriari2015taking,garnett2023bayesian} is designed for this setting. It fits a probabilistic surrogate to the evaluations collected so far and selects the next point by maximizing a cheap acquisition function built on this surrogate. BO has become a standard tool for problems of the form~\eqref{eq:main-prob} in machine learning~\cite{snoek2012practical}, chemistry~\cite{shields2021bayesian,guo2023bayesian} and aerospace design~\cite{priem2020efficient,zhang2023multi}. Its best-known instance is efficient global optimization (EGO)~\cite{jones1998efficient}, which pairs a Gaussian process (GP)~\cite{rasmussen2006gpml} with the expected improvement (EI)~\cite{mockus1974bayesian}. Diouane \textit{et~al.}~\cite{diouane2023trego} embedded EGO in a trust-region scheme called TREGO. Each TREGO iteration starts with a global phase, which performs one EGO iteration over $\Omega$. If this phase fails to achieve a sufficient decrease, a local phase maximizes EI inside a trust region around the incumbent. The sufficient-decrease condition, which also governs the trust-region radius, gives TREGO its convergence guarantees.

However, TREGO handles neither high dimension nor noise well. As $d$ grows, the surrogate requires far more data to cover the domain, and the acquisition subproblem becomes harder to solve. Moreover, EI rewards a large predictive variance, which is a pointwise measure of uncertainty and is largest far from the observed data. In high dimension, the points farthest from the data typically lie near the boundary of $\Omega$. Maximizing EI over the whole domain, as EGO and the global phase of TREGO do, therefore drives candidates toward the corners, where little is learned. This behavior is known as the boundary issue~\cite{siivola2018correcting,oh2018bock}. With noise, the convergence guarantees of TREGO no longer apply, because they assume exact function values, which the oracle does not provide.

Convergence guarantees under noise have been established outside BO, for methods whose step acceptance relies, as in TREGO, on a sufficient-decrease condition. Chen \textit{et~al.}~\cite{chen2018stochastic} proposed STORM, a trust-region method based on random models and on noisy function estimates that are sufficiently accurate with a fixed probability. They proved that STORM converges to a stationary point with probability one. Blanchet \textit{et~al.}~\cite{blanchet2019convergence} then used a supermartingale-based framework to bound the expected number of iterations STORM needs to reach an approximate stationary point. Paquette and Scheinberg~\cite{paquette2020stochastic} applied this probabilistic framework to stochastic line-search methods. Diouane \textit{et~al.}~\cite{diouane2022globally} extended it to evolution strategies, with a sufficient-decrease condition on noisy estimates and an extreme-barrier approach to constraints, and proved convergence to a first-order stationary point with probability one. More recently, Ding \textit{et~al.}~\cite{ding2026non} proposed a nonmonotone stochastic direct-search method and bounded its expected number of iterations with the stopping-time framework of~\cite{blanchet2019convergence}. All these methods, however, require the objective function to have a Lipschitz continuous gradient. For nonsmooth problems, Audet \textit{et~al.}~\citep{audet2019stomads} introduced StoMADS, which extends the mesh adaptive direct search (MADS) algorithm~\cite{AudetDennis2006} and its convergence analysis to noisy estimates of the objective function. StoMADS requires only local Lipschitz continuity of the objective. Using the density of the generated directions and martingale theory, its authors proved convergence to a Clarke stationary point with probability one.

Within BO, most methods designed for high dimension are heuristic and come with no convergence guarantees. Several of them are built around trust regions~\cite{eriksson2019scalable,namura2025regional,tang2026rethinking}, as is the local phase of TREGO. A trust region keeps candidates away from the boundary, and the surrogate only has to be accurate inside it, which reduces the amount of data needed for a reliable fit. Some instead assume structure in the objective, such as additivity~\cite{kandasamy2015high,bardou2023relaxing,ziomek2023random} or a low-dimensional active subspace~\cite{eriksson2021high,papenmeier2022increasing}.

Other works modify the acquisition step. Batch methods select several points per iteration rather than one, to improve the exploration--exploitation balance and to use parallel computing resources~\cite{ginsbourger2010kriging,snoek2012practical,daulton2021parallel}. Such batches are commonly produced by multistart strategies applied to the acquisition function. An alternative is to keep the posterior mean and variance of the surrogate as two separate objectives, rather than combine them into a single acquisition function, and to solve the resulting bi-objective problem~\cite{bischl2014moi,vzilinskas2019bi}. The Pareto front then offers a range of trade-offs between the two objectives and yields a batch naturally. Carciaghi \textit{et~al.}~\cite{carciaghi2025bi} built on this idea. They reconstruct the Pareto front with the Non-dominated Sorting Memetic Algorithm (NSMA)~\cite{lapucci2023memetic}, which combines evolutionary operators with gradient-based refinement, and extract the batch by clustering, either in the input space or in the objective space. Their method is presented as a batch acquisition strategy and is explicitly studied for $d \geq 20$. 

Each of these BO approaches has weaknesses. Trust-region solvers do not separate global exploration from local exploitation into two independent mechanisms. Methods based on structural assumptions fail, or waste evaluations, when the effective dimension is high. Multi-objective methods still use the predictive variance as their measure of uncertainty, so they remain exposed to the boundary issue. Clustering in the input space moves some candidates back toward the interior, but only as a correction applied after they are generated. In addition, most of these works assume noise-free evaluations. These observations raise the question that this paper addresses. Can a trust-region BO method keep its convergence guarantees with noisy evaluations and remain competitive in high dimension?

Our contributions are threefold. First, we generalize TREGO into a trust-region framework for noisy BO. Its definition and convergence analysis cover noisy function estimates that are sufficiently accurate with sufficiently high probability. They also cover iterations in which the trust region contains no feasible point, a case that TREGO does not address. As in StoMADS, the analysis requires $f$ to be locally Lipschitz continuous rather than to have a Lipschitz continuous gradient. Second, we introduce STREGO, a scalable implementation of this framework that couples a bi-objective global infill criterion with a purely exploitative local phase. The two objectives of the global criterion are the posterior mean and the integrated variance reduction, which measures how much evaluating a candidate would reduce the predictive variance over the whole domain. A point near the boundary is therefore no longer favored merely because it is far from the observed data. The local phase minimizes, inside the trust region, the posterior mean of a separate GP trained only on the evaluated points near the incumbent, so that its lengthscales adapt to this region rather than to the whole domain. Third, we test STREGO on synthetic problems of dimension $4$ to $100$ and on real-world problems of dimension $16$ to $124$, with and without noise. A first experiment shows that, on three of the four synthetic functions tested, the integrated variance reduction generally leads to better solutions than the predictive variance. We then compare STREGO with TREGO and with TURBO~\cite{eriksson2019scalable}, BAXUS~\cite{papenmeier2022increasing} and NSMA-BO~\cite{carciaghi2025bi}, three state-of-the-art solvers for high-dimensional BO. On the synthetic problems, STREGO obtains the best mean solution quality in every configuration tested. From dimension $50$ on, it is also cheaper than all the other solvers, except NSMA-BO in the noise-free setting. On four of the six real-world problems, STREGO finds the best solution, alone on three and tied with TREGO on the fourth.

The remainder of the paper is organized as follows. Section~\ref{sec:preliminaries} recalls EGO and TREGO, with a focus on the GP posterior. Section~\ref{sec:Framework} presents the framework and its convergence analysis. Section~\ref{sec:STREGO} introduces STREGO, motivates the use of the integrated variance reduction and details the global and local phases. Section~\ref{sec:numerical} reports the numerical experiments. Finally, Section~\ref{sec:conclusion} concludes the paper and outlines directions for future research. An implementation of STREGO is available at \url{https://github.com/Mayousk/STREGO}.

\section{Preliminaries: EGO and TREGO}
\label{sec:preliminaries}
In this section, we first recall EGO, with a focus on the GP surrogate and the EI criterion and then describe the mechanism of TREGO. 
\subsection{EGO}
EGO~\cite{jones1998efficient} is an instance of the BO framework in which a GP surrogate is trained on a growing design of experiments (DoE) and the next candidate is selected by maximizing the EI infill criterion. The DoE $\mathcal{D}_n = (\mathcal{X}_n, \mathcal{Y}_n)$ consists of $n$ design points $\mathcal{X}_n = \{x_1, \dots, x_n\} \subset \Omega$ and their corresponding evaluations $\mathcal{Y}_n = \{y_1, \dots, y_n\}$ and the GP learns the mapping between the two. Formally, the objective function is modeled as a Gaussian process, $f \sim \mathcal{GP}(m, k)$, with constant prior mean $m$ and covariance kernel $k$, whose hyperparameters are re-estimated from the collected data as the optimization progresses. Common choices are stationary kernels such as the squared-exponential kernel and the Mat\'ern family. These kernels depend on the inputs only through the scaled distance $r(x, x') = \big(\sum_{i=1}^{d} (x_i - x'_i)^2 / \ell_i^2\big)^{1/2}$. The lengthscale $\ell_i > 0$ controls how quickly the correlation between two points decays along the $i$-th coordinate. A small $\ell_i$ lets $f$ vary rapidly in that direction, while a large $\ell_i$ makes it nearly flat.

Let $K \in \mathbb{R}^{n \times n}$ be the matrix with entries $K_{ij} = k(x_i, x_j)$ and $k_n(x) = [k(x, x_1), \dots, k(x, x_n)]^\top$. Conditioning on $\mathcal{D}_n$ yields a Gaussian posterior $f(x) \mid \mathcal{D}_n \sim \mathcal{N}\big(\mu_n(x), \sigma^2_n(x)\big)$ whose mean and variance are given in closed form by:
\begin{align}
    \mu_n(x)      &= m + k_n(x)^\top \left(K + \tau^2 I\right)^{-1} \left(\mathcal{Y}_n - m\mathbf{1}\right), \label{eq:post_mean}\\
    \sigma_n^2(x) &= k(x, x) - k_n(x)^\top \left(K + \tau^2 I\right)^{-1} k_n(x), \label{eq:post_var}
\end{align}
where $\tau^2 \geq 0$ is the variance of the Gaussian observation noise ($\tau^2 = 0$ for noiseless evaluations), $I$ is the $n \times n$ identity matrix and $\mathbf{1}$ is the $n$-dimensional vector of ones. The mean $\mu_n(x)$ is the model's estimate of $f(x)$, while the standard deviation $\sigma_n(x)$ quantifies the model's uncertainty about this estimate.

Given the current best value $y^{\min}_n = \min_{1 \le j \le n} y_j$, EI measures the expected amount by which a candidate $x$ improves upon it under the GP posterior,
$\mathrm{EI}_n(x) = \mathbb{E}\!\left[\max\big(y^{\min}_n - f(x),\, 0\big) \,\middle|\, \mathcal{D}_n\right]$.
EI admits the closed form
\begin{equation}
    \mathrm{EI}_n(x) =
    \begin{cases}
        \big(y^{\min}_n - \mu_n(x)\big)\,\Phi\big(z_n(x)\big) + \sigma_n(x)\,\phi\big(z_n(x)\big) & \text{if } \sigma_n(x) > 0,\\[4pt]
        \max\big(y^{\min}_n - \mu_n(x),\, 0\big) & \text{if } \sigma_n(x) = 0,
    \end{cases}
    \label{eq:ei}
\end{equation}
where $z_n(x) = \frac{y^{\min}_n - \mu_n(x)}{\sigma_n(x)} $ and $\Phi$ and $\phi$ denote the standard normal cumulative distribution and density functions. The first term rewards candidates with a low predicted mean (exploitation) and the second rewards candidates with high uncertainty (exploration).

\subsection{TREGO}
\label{sec:trego}
Building on the EGO framework and on the search-poll paradigm of direct-search methods~\cite{AuHa2026}, TREGO~\cite{diouane2023trego} embeds EGO within a trust-region scheme that comes with theoretical convergence guarantees. At each iteration $k$, TREGO maintains an incumbent $x_k$, its function value $f(x_k)$ and a trust-region radius $\delta_k > 0$. The iteration starts with a global phase and a local phase is performed only if the global phase fails to achieve a sufficient decrease.

The global phase performs one EGO iteration over the entire domain $\Omega$: EI is maximized over $\Omega$ and the resulting candidate $x^{\glb}_{k}$ is evaluated, yielding $f(x^{\glb}_{k})$. This candidate is accepted only if it achieves a sufficient decrease with respect to the incumbent, namely
\begin{equation}
    f(x^{\glb}_{k}) \leq f(x_k) - \frac{\kappa}{2}\,\delta_k^2,
    \label{eq:sufficient-decrease-TREGO-global}
\end{equation}
where $\kappa > 0$ is a user-defined constant. If \eqref{eq:sufficient-decrease-TREGO-global} holds, the global phase is successful, the incumbent is updated to $x_{k+1} = x^{\glb}_{k}$, $\delta_{k+1}=\gamma\delta_{k}$ (with $\gamma > 1$) and the local phase is skipped. Otherwise, the iteration proceeds to the local phase.

The local phase follows the same mechanism, except that EI is maximized over a trust region around the incumbent given by $\overline{\Omega}_k$:
\begin{equation}
    \overline{\Omega}_k = \left\{ x \in \Omega \;\middle|\; r_{\min}\,\delta_k \leq \|x - x_k\| \leq r_{\max}\,\delta_k \right\},
    \label{eq:trust-region-TREGO}
\end{equation}
where $0 < r_{\min} < r_{\max}$ are user-defined constants.
The resulting local candidate $x^{\lcl}_{k} \in \overline{\Omega}_k$ is evaluated and tested against the local sufficient decrease condition 
\begin{equation}
    f(x^{\lcl}_{k}) \leq f(x_k) - \frac{\kappa}{2}\,\delta_k^2.
    \label{eq:sufficient-decrease-TREGO-local}
\end{equation}
If this condition is satisfied, the local phase is a success, $x_{k+1} = x^{\lcl}_{k}$ and the radius is expanded as $\delta_{k+1} = \gamma\,\delta_k$. Otherwise, the incumbent is kept, $x_{k+1} = x_k$ and the radius is shrunk as $\delta_{k+1} = \beta\,\delta_k$ with $0 < \beta < 1$.

The general framework of TREGO allows each phase to consist of several consecutive EGO steps: $\mathrm{G}$ steps over $\Omega$ in the global phase and $\mathrm{L}$ steps within $\overline{\Omega}_k$ in the local phase. The sufficient-decrease condition is checked only at the end of each phase, once all $\mathrm{G}$ (respectively $\mathrm{L}$) steps have been performed. The resulting scheme is illustrated in Figure~\ref{fig:trego}.

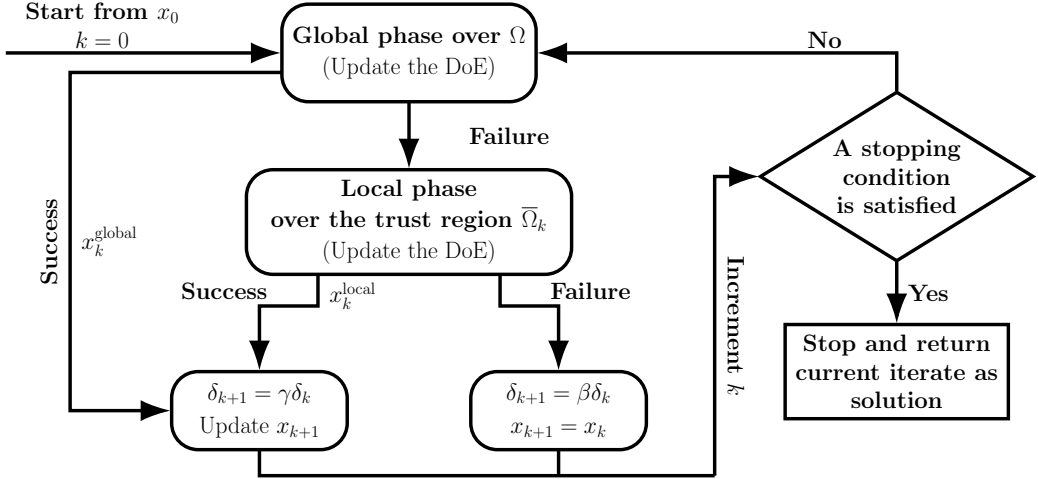
\begin{figure}[!htbp]
    \centering
    \begin{tikzpicture}[
        scale=0.43, transform shape,
        font=\bfseries\huge,
        line width=1.3pt,
        >={Latex[length=4mm,width=2.8mm]},
        box/.style  = {draw, rectangle, align=center, inner sep=9pt},
        rbox/.style = {box, rounded corners=10pt},
        dec/.style  = {draw, diamond, align=center, inner sep=0pt, aspect=1.5},
        obox/.style = {rbox, minimum width=5.4cm, minimum height=2.4cm},
    ]

        \node[rbox, minimum width=7.5cm, minimum height=3cm] (global) at (0,0)
            {Global phase over $\Omega$\\[2pt] \mdseries (Update the DoE)};

        \node[rbox, minimum width=10cm, minimum height=2.8cm] (local) at (0,-5.2)
            {Local phase\\[2pt] over the trust region $\overline{\Omega}_k$\\[2pt] \mdseries (Update the DoE)};

        \node[obox] (succ) at (-4.6,-11)
            {$\delta_{k+1}=\gamma\delta_k$\\[4pt] \mdseries Update $x_{k+1}$};
        \node[obox] (fail) at ( 4.6,-11)
            {$\delta_{k+1}=\beta\delta_k$\\[4pt] $x_{k+1}=x_k$};

        \node[dec, minimum width=8.4cm, minimum height=5.2cm] (test) at (15,-3.8)
            {A stopping\\ condition\\ is satisfied};
        \node[box, minimum width=6.4cm, minimum height=2.4cm] (stop) at (15,-9.8)
            {Stop and return\\ current iterate as\\ solution};

        \draw[->] (-12.4,0) -- node[pos=0.35, above, align=center]
            {Start from $x_0$\\ $k=0$} (global.west);

        \draw[->] (global.south) -- node[right=17mm] {Failure} (local.north);

       \draw[->] ([yshift=-6mm]global.west) -| ([xshift=-3.1cm]succ.west)
    node[pos=0.75, xshift=-2.5mm, rotate=90, anchor=south] {Success}
    node[pos=0.75, right=2mm] {$x^{\mathrm{global}}_{k}$}
    -- (succ.west);

        \draw[->] ([xshift=-2.8cm]local.south) -- ++(0,-0.95)
            node[midway, left=14mm] {Success}
            node[midway, right=2mm, yshift=-1.5mm] {$x^{\mathrm{local}}_{k}$}
            -| (succ.north);

        \draw[->] ([xshift=2.8cm]local.south) -- ++(0,-0.95)
            node[midway, right=14mm] {Failure}
            -| (fail.north);

        \draw (succ.south) |- (9.4,-13);
        \draw (fail.south) -- (fail.south |- 0,-13);
        \draw[->] (9.4,-13) |- node[pos=0.25, xshift=2.5mm, rotate=-90, above] {Increment $k$} (test.west);

        \draw[->] (test.north) |- node[pos=0.6, above] {No} (global.east);
        \draw[->] (test.south) -- node[right=2mm] {Yes} (stop.north);
    \end{tikzpicture}
    \caption{Overview of TREGO, adapted from \cite{diouane2023trego} (here $0<\beta < 1 <\gamma $).}
    \label{fig:trego}
\end{figure}
\section{A trust-region framework for noisy BO}
\label{sec:Framework}
In this section, we introduce a general stochastic trust-region framework for noisy BO. The framework is designed to be general enough to support different implementations. Its convergence analysis is given in Section~\ref{sec: convergence_proof}. 

\subsection{The framework}
\label{app:framework}

We extend TREGO to the noisy setting, keeping the global and local steps unspecified. The way candidates are generated in each phase is left as a design choice and only the structure that alternates between the two phases and decides when a step is accepted is fixed.

Unlike TREGO, the proposed framework does not require exact function values at the global and local candidates. Instead, $f^{\glb}_{k}$ and $f^{\lcl}_{k}$ are only assumed to be estimates of $f(x^{\glb}_{k})$ and $f(x^{\lcl}_{k})$ and the framework leaves the estimation procedure open. Any estimator can therefore be used, ranging from a single noisy evaluation to a sample average over repeated evaluations at the same point, or the posterior mean of the surrogate. Furthermore, the incumbent value $f_{k} \approx f(x_k)$ is re-estimated at each iteration. 

Once a global candidate is generated, the global sufficient-decrease condition given by Equation~\eqref{eq:sufficient-decrease-TREGO-global} becomes:
\begin{equation}
    f^{\glb}_{k} \leq f_{k} - \frac{\kappa}{2}\,\delta_k^2,
    \label{eq:sufficient-decrease-global-framework}
\end{equation}
 If this condition is satisfied, the update proceeds as described in Section~\ref{sec:trego}. Otherwise, the algorithm moves to the local phase.

The local phase is performed on a trust region $\overline{\Omega}_k$ defined as an annulus around the current incumbent $x_k$ as stated in \eqref{eq:trust-region-TREGO}.
However, in the constrained case where $\Omega \subsetneq \mathbb{R}^{d}$, this set
may be empty. This might happen for instance when the step size $\delta_k$ is too large relative to
the extent of the feasible domain around $x_k$. Figure~\ref{fig:tr-cases}
illustrates the possible configurations. In cases~(a) and~(b), $\overline{\Omega}_k$ is
nonempty and the local phase is well defined. These are the only cases covered by TREGO~\cite{diouane2023trego}. Case~(c), where $\overline{\Omega}_k=\emptyset$, is not addressed in
the original framework. For this configuration, the local phase cannot be performed, so the iteration is
not well defined and its convergence analysis does not apply. To close this gap, we introduce an additional (locally) empty iteration,
in which the incumbent is kept and the step size is reduced. Since an empty
trust region reveals nothing about $f$ and is purely a geometric issue, the
step size is reduced only mildly. The update rule is therefore given by
$x_{k+1}=x_k$ and $\delta_{k+1}=\beta_0\delta_k$ with $\beta_0<1$.

\begin{figure}[!h]
  \centering
  \begin{subfigure}[t]{0.32\linewidth}
    \centering
    \parbox[c][3cm][c]{\linewidth}{\centering
    \begin{tikzpicture}[scale=0.68]
      \draw[omegaset] (0,0) rectangle (4.6,3.6);
      \node at (0.3,3.3) {$\Omega$};
      \coordinate (xk) at (2.3,1.8);
      \fill[ringfill, even odd rule] ($(xk)+(-1.2,-1.2)$) rectangle ++(2.4,2.4)
                                     ($(xk)+(-0.6,-0.6)$) rectangle ++(1.2,1.2);
      \draw[ring] ($(xk)+(-1.2,-1.2)$) rectangle ++(2.4,2.4)
                  ($(xk)+(-0.6,-0.6)$) rectangle ++(1.2,1.2);
      \node[incumbent, label={[font=\scriptsize, inner sep=0.5pt]above:$x_k$}] at (xk) {};
    \end{tikzpicture}}
    \caption{ $\overline{\Omega}_k\neq \emptyset$ (annulus inside).}
    \label{fig:tr-inside}
  \end{subfigure}
  \hfill
  \begin{subfigure}[t]{0.34\linewidth}
    \centering
    \parbox[c][3cm][c]{\linewidth}{\centering
    \begin{tikzpicture}[scale=0.68]
      \draw[omegaset] (0,0) rectangle (4.6,3.6);
      \node at (0.3,0.3) {$\Omega$};
      \coordinate (xk) at (3.9,2.9);
      \begin{scope}
        \clip (0,0) rectangle (4.6,3.6);
        \fill[ringfill, even odd rule] ($(xk)+(-1.2,-1.2)$) rectangle ++(2.4,2.4)
                                       ($(xk)+(-0.6,-0.6)$) rectangle ++(1.2,1.2);
      \end{scope}
      \draw[ring] ($(xk)+(-1.2,-1.2)$) rectangle ++(2.4,2.4)
                  ($(xk)+(-0.6,-0.6)$) rectangle ++(1.2,1.2);
      \node[incumbent, label={[font=\scriptsize, inner sep=0.5pt]below:$x_k$}] at (xk) {};
    \end{tikzpicture}}
    \caption{$\overline{\Omega}_k\neq \emptyset$ (annulus partly outside).}
    \label{fig:tr-partial}
  \end{subfigure}
  \hfill
  \begin{subfigure}[t]{0.32\linewidth}
    \centering
    \parbox[c][3cm][c]{\linewidth}{\centering
    \begin{tikzpicture}[scale=0.68]
      \coordinate (xk) at (1.1,0.8);
      \draw[badring] ($(xk)+(-2.2,-2.2)$) rectangle ++(4.4,4.4)
                     ($(xk)+(-1.6,-1.6)$) rectangle ++(3.2,3.2);
      \draw[omegaset] (0,0) rectangle (2.2,1.6);
      \node at (0.3,1.3) {$\Omega$};
      \node[incumbent, label={[font=\scriptsize, inner sep=0.5pt]above:$x_k$}] at (xk) {};
    \end{tikzpicture}}
    \caption{$\overline{\Omega}_k=\emptyset$.}
    \label{fig:tr-large}
  \end{subfigure}
  \caption{An $\real^2$ illustration of the annular trust region $\overline{\Omega}_k$ (using the infinity norm) for different
  step sizes and incumbent positions $x_k \in \Omega$.}
  \label{fig:tr-cases}
\end{figure}
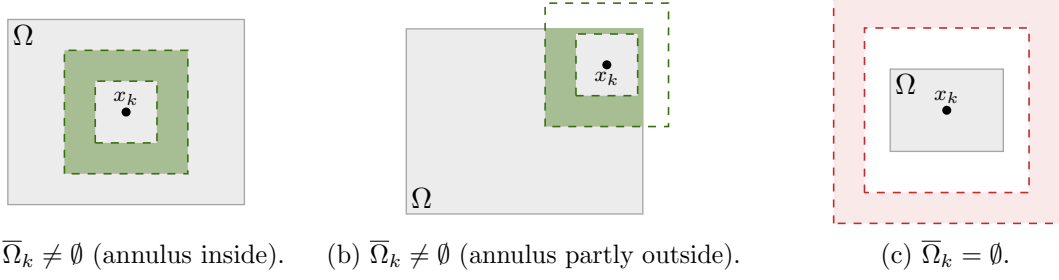

When the trust region is nonempty ($\overline{\Omega}_k\neq\emptyset$), a local candidate
$x^{\mathrm{\lcl}}_k$ is generated within $\overline{\Omega}_k$. Its estimate $f^{\lcl}_{k}$ is then tested against the local sufficient-decrease condition:
\begin{equation}
    f^{\lcl}_{k} \leq f_{k} - \frac{\kappa}{2}\,\delta_k^2.
    \label{eq:sufficient-decrease-local-framework}
\end{equation}
If this condition is satisfied, the iteration is successful, $x_{k+1}=x^{\lcl}_{k}$ and $\delta_{k+1}=\gamma\delta_k$. In the original TREGO framework, if $f^{\lcl}_{k}$ does not satisfy this condition, the radius is shrunk by a parameter $\beta$. However, in the noisy setting, the estimates remain polluted by noise to some extent, so the outcome of the test is itself uncertain. Shrinking the trust region based on noisy estimates can therefore lead to a premature collapse. Instead, we distinguish two types of unsuccessful local iterations: certain and uncertain ones. If $f_{k}^{\lcl}$ sufficiently violates Equation~\eqref{eq:sufficient-decrease-local-framework} meaning,
\begin{equation}
    f^{\lcl}_{k} \geq f_{k} + \frac{\kappa}{2}\,\delta_k^2,
    \label{eq:certain-failure}
\end{equation}
the iteration is declared a certain unsuccessful iteration, $x_{k+1}=x_k$ and the radius is shrunk by $\beta_2$ such that $\delta_{k+1}=\beta_2\delta_k$. Otherwise, the iteration is considered uncertainly unsuccessful, $x_{k+1}=x_k$ and the radius is shrunk by $\beta_1$ with $\beta_2\le\beta_1\leq\beta_0<1$. 

Commonly used in stochastic settings~\cite{blanchet2019convergence,audet2019stomads}, this procedure defines a tolerance interval within which the iteration is declared uncertainly unsuccessful and the trust region undergoes a mild shrinkage governed by $\beta_1$. Beyond this interval, the iteration is declared certainly unsuccessful and the radius is shrunk with the smaller parameter $\beta_2$. The complete procedure is detailed in Algorithm~\ref{alg:framework}.

\begin{algorithm}[!h]
    \footnotesize
    \caption{\bf A trust-region framework for noisy BO.}
    \label{alg:framework}
    \SetAlgoLined
    \LinesNotNumbered

    \KwData{
        Initialize the step size parameter $\delta_0>0$, sample an initial DoE $\mathcal{D}=(\mathcal{X},\mathcal{Y})$ of size $N_0$ with $\mathcal{X}=\{x_i\}_{i=1}^{N_0}\subset\Omega$, $\mathcal{Y}=\{f_i \approx f(x_i)\}_{i=1}^{N_0}$ and $x_0 \in \mathcal{X}$; Choose constants $0<\beta_2 \leq \beta_1 \leq \beta_0 < 1 < \gamma$ such that $\gamma\beta_2 \geq 1$.
        Choose constants $\kappa>0$, $r_{\max}> r_{\min}>0$ and set $k=0$.
    }

    \While{the stopping criterion is not satisfied}{
        \textbf{Step 0 (estimate at the incumbent):}

        Compute an estimate $f_{k} \approx f(x_k)$.

        \textbf{Step 1 (global BO phase over $\Omega$):}

        Generate a global candidate $x^{\glb}_{k} \in \Omega$ and an estimate $f^{\glb}_{k} \approx f(x^{\glb}_{k})$.

        \textbf{Step 2 (global sufficient-decrease test):}

        \eIf{$f^{\glb}_{k} \leq f_{k} - \frac{\kappa}{2}\delta_k^2$}{
            The iteration is (globally) successful. Set
            $x_{k+1}=x^{\glb}_{k}$ and $\delta_{k+1}=\gamma\delta_k$.
        }{
            Define the local trust region
            $\overline{\Omega}_k = \left\{ x \in \Omega \;\middle|\; r_{\min}\,\delta_k \leq \|x - x_k\| \leq r_{\max}\,\delta_k \right\}.$

            \eIf{$\overline{\Omega}_k \neq \emptyset$}{

                \textbf{Step 3 (local BO phase over the trust region $\overline{\Omega}_k$):}

                Generate a local candidate $x^{\lcl}_{k} \in \overline{\Omega}_k$ and an estimate $f^{\lcl}_{k} \approx f(x^{\lcl}_{k})$.

                \textbf{Step 4 (local sufficient-decrease test):}

                \eIf{$f^{\lcl}_{k} \leq f_{k} - \frac{\kappa}{2}\delta_k^2$}{
                    The iteration is (locally) successful. Set
                    $x_{k+1}=x^{\lcl}_{k}$ and $\delta_{k+1}=\gamma\delta_k$.
                }{
                    \eIf{$f^{\lcl}_{k} \geq f_{k} + \frac{\kappa}{2}\delta_k^2$}{
                        The iteration is a certain unsuccessful iteration. Set
                        $x_{k+1}=x_k$ and $\delta_{k+1}=\beta_2\delta_k$.
                    }{
                        The iteration is an uncertain unsuccessful iteration. Set
                        $x_{k+1}=x_k$ and $\delta_{k+1}=\beta_1\delta_k$.
                    }
                }
            }{
                The iteration is (locally) empty. Set
                $x_{k+1}=x_k$ and $\delta_{k+1}=\beta_0\delta_k$.
            }
        }

        Increment $k$\;
    }
\end{algorithm}

\subsection{Convergence analysis}
\label{sec: convergence_proof}


Throughout this section, the objective function $f$ is assumed to be accessed only through stochastic estimators. As a
consequence, all the quantities generated by Algorithm~\ref{alg:framework} are considered as random variables defined on a probability space $(\mathcal{E},\mathcal{G},\mathbb{P})$: the iterate $X_k$, the step-size parameter
$\Delta_k$, the direction $Q_k$, the step $S_k=\Delta_kQ_k$ and the estimates $F_k$ and $F_k^s$ of $f$ at $X_k$ and at the trial point $X_k+S_k$ defined below. Given a sample $\omega\in\mathcal{E}$, we denote their realizations by lowercase letters:
$x_k=X_k(\omega)$, $\delta_k=\Delta_k(\omega)$, $q_k=Q_k(\omega)$, $s_k=S_k(\omega)$,
$f_k=F_k(\omega)$ and $f_k^s=F_k^s(\omega)$. We denote by $\mathcal{F}_{k-1}$ the
$\sigma$-algebra generated by all the random quantities produced by
Algorithm~\ref{alg:framework} up to iteration $k-1$, in particular by
$F_0,F^{s}_0,\ldots,F_{k-1},F_{k-1}^{s}$, with $\mathcal{F}_{-1}=\sigma(x_0)$. Thus, $\{\mathcal{F}_k\}_{k\ge-1}$ is a
filtration and $X_k$ and $\Delta_k$ are $\mathcal{F}_{k-1}$-measurable.


Note that Algorithm~\ref{alg:framework} can be
interpreted as generating, at each iteration $k$, a trial point
\begin{equation}\label{xs}
x^{s}_{k}=
\begin{cases}
x^{\lcl}_{k} & \text{if the global phase is declared unsuccessful and } \overline{\Omega}_k\neq\emptyset,\\
x^{\glb}_{k} & \text{otherwise.}
\end{cases}
\end{equation}
The point
$x^s_{k}$ is always involved in the last sufficient-decrease test performed at
iteration $k$. By construction, $x^s_k\in\Omega$ for all $k$. Since $x_0\in\Omega$ and
$x_{k+1}\in\{x_k,\,x^s_k\}$, all the iterates also belong to $\Omega$, so that $f(X_k)$ and
$f(X_k+S_k)$ are well defined. We denote by $\mathcal{U}$ the (random) set
of indices $k$ for which the global phase is declared unsuccessful and
$\overline{\Omega}_k\neq\emptyset$ (the iterations where the local phase over $\overline{\Omega}_k$ is
performed) and by $\mathcal{U}_{\emptyset}$ the (random) set of indices for which the
global phase is declared unsuccessful and $\overline{\Omega}_k=\emptyset$.
The stochastic convergence analysis is inspired by what is done in~\cite{paquette2020stochastic,chen2018stochastic,audet2019stomads}.

In order to obtain convergence guarantees for Algorithm~\ref{alg:framework}, we require
the estimates of $f$ to be sufficiently accurate with a suitable probability. This is formalized in Definition~\ref{defi:probaccfun}.
\begin{definition} \label{defi:probaccfun}
Given constants $\varepsilon_f>0$ and $p \in (0,1]$, the sequence of the random quantities
$F_k$ and $F_k^{s}$ is called $p$-probabilistically $\varepsilon_f$-accurate, for
corresponding sequences $\{X_k\}$, $\{S_k\}$, $\{\Delta_k\}$, if the event
\begin{align*}
T^f_k \; \stackrel{\Delta}{=} \; \big\{ \left|F_k - f(X_k) \right| \; \le \;
\varepsilon_f \Delta_k^{2}
~\mathrm{and}~ \left|F_k^{s} - f(X_{k}+S_k) \right| \; \le \;
\varepsilon_f \Delta_k^{2} \big\}
\end{align*}
satisfies the condition
\[
p_k^* \; \stackrel{\Delta}{=} \;\mathbb{P}\left(~T^f_k~~\big| \mathcal{F}_{k-1}  \right) \geq p . \qquad 
\]
\end{definition}

We derive a convergence analysis of Algorithm~\ref{alg:framework} under the following assumptions.

\begin{assumption} \label{ass:f}
$f$ is bounded below by $f_{\mathrm{low}}$ on $\Omega$.
\end{assumption}


\begin{assumption} \label{ass:ProbAccEstimates:obj}
The sequence of random objective function estimates $\{F_k, F_k^{s}\}_k$ generated by Algorithm~\ref{alg:framework} satisfies the following two conditions.
\begin{enumerate}
\item[(1)] The sequence $\{F_k, F_k^{s}\}_k$ is $p$-probabilistically
$\varepsilon_f$-accurate for some $p \in (\frac{1}{2},1]$ and
$\varepsilon_f \in (0,\frac{\kappa}{8}]$, where $\kappa$ is the constant used in Algorithm~\ref{alg:framework}.
\item[(2)] There exists $\varepsilon_v>0$ such that the sequence of estimates
$\{F_k, F_k^{s}\}_k$ satisfies the following $\varepsilon_v$-variance condition for all $k \ge 0$,
\begin{align*}
\mathbb{E}\left(\left|F_k - f(X_k) \right|^2 \big| \mathcal{F}_{k-1}  \right)
&\;\le\; \varepsilon_v^2 \Delta_k^4
\quad \mathrm{and} \\
\mathbb{E}\left(\left|F_k^{s} - f(X_{k}+S_k) \right|^2 \big| \mathcal{F}_{k-1}  \right)
&\;\le\; \varepsilon_v^2 \Delta_k^4.
\end{align*}
\end{enumerate}
\end{assumption}
For the sake of our analysis, we introduce the following (random) Lyapunov function
\begin{equation}
\Phi_k \ = \ \nu f(X_k) + (1 - \nu) \Delta_k^2,
\label{eq:Lyapunov}
\end{equation}
where $\nu \in (0, 1)$ satisfies
\begin{equation}
\frac{\nu}{1 - \nu} \ge \frac{8}{\kappa}\max\{1 - \beta_0^{2},\gamma^2 -1\}.
\label{eq:nu}
\end{equation}
Since $f(x_0)$ is finite, $f \geq f_{\mathrm{low}}$ on $\Omega$ and, on successful iterations, $f(X_{k+1}) \leq f(X_k) + |F_k^s - f(X_k + S_k)| + |F_k - f(X_k)|$, where the error terms are integrable by Assumption~\ref{ass:ProbAccEstimates:obj}(2) and $\Delta_k \leq \gamma^k \delta_0$, an induction on $k$ shows that $\Phi_k$ is integrable, so all conditional expectations below are well defined.
 Consider a realization of Algorithm~\ref{alg:framework} and let $\phi_k$ be the corresponding realization of $\Phi_k$. We denote by $S$ the event ``the iteration $k$ is successful'', by $\overline{S}$ its complement and by $\mathbf{1}_S$ (respectively $\mathbf{1}_{\overline{S}}$) the corresponding indicator function. 

\begin{lemma} \label{lm:1}
Let Assumptions~\ref{ass:f} and~\ref{ass:ProbAccEstimates:obj} hold. Then, for every iteration $k$, one has
\begin{align*}
\max \Big[ &\mathbb{E}\left(\mathbf{1}_{\overline{T^f_k}} \left|F_k - f(X_k) \right| \big| \mathcal{F}_{k-1}  \right), \\
&\mathbb{E}\left(\mathbf{1}_{\overline{T^f_k}} \left|F_k^{s} - f(X_{k}+S_k) \right| \big| \mathcal{F}_{k-1}  \right) \Big]
\le (1-p_k^*)^{1/2} \varepsilon_v \Delta_k^2 \le (1-p)^{1/2} \varepsilon_v \Delta_k^2,
\end{align*}
where $\overline{T^f_k}$ denotes the event where the objective function estimates $(F_k, F_k^{s})$ are inaccurate and $\mathbf{1}_{\overline{T^f_k}}$ is the indicator function of the event set $\overline{T^f_k}$.
\end{lemma}

\begin{proof}
The result follows from the conditional Cauchy--Schwarz inequality, $\mathbb{E}(\mathbf{1}_{\overline{T^f_k}}|\mathcal{F}_{k-1}) = 1-p_k^*$ and Assumption~\ref{ass:ProbAccEstimates:obj} (2) as done in~\cite{paquette2020stochastic}.
\end{proof}

The next lemma quantifies the decrease of the Lyapunov function on the event where the
estimates are accurate.
\begin{lemma} \label{lm:2}
Let Assumptions~\ref{ass:f} and~\ref{ass:ProbAccEstimates:obj} hold and let $\nu$ satisfy~\eqref{eq:nu}. Then, for every iteration $k$, one has
\[
\mathbb{E}\left(\mathbf{1}_{T^f_k} \left(\Phi_{k+1} - \Phi_k\right) \big| \mathcal{F}_{k-1}  \right)
\le - (1 - \nu)(1- \beta_0^{2} )p_k^*\Delta_k^2 \le - (1 - \nu)(1- \beta_0^{2} )p\,\Delta_k^2,
\]
where $\mathbf{1}_{T^f_k}$ is the indicator function of the event set $T^f_k$.
\end{lemma}

\begin{proof}
Consider a realization of a given iteration $k$ of Algorithm~\ref{alg:framework} for which the objective function estimates $\{f_k, f_k^{s}\}$ are $\varepsilon_f$-accurate.
If the iteration is unsuccessful, then $x_{k+1} = x_k$ and $\delta_{k+1} \le \max\{\beta_2,\beta_1,\beta_0\}\,\delta_k = \beta_0\,\delta_k$, since $\beta_2\le\beta_1\leq\beta_0<1$, which leads to
\begin{equation} \label{eq:DeltaPhiUnsuccessful}
\phi_{k+1} - \phi_{k} \le (1 - \nu)(\beta_0^{2} - 1)\delta_k^2 := b_1 < 0.
\end{equation}
Otherwise, if the iteration is successful, then one has
$x_{k+1} = x^{s}_{k} = x_k + \delta_k q_k$,
$f_k^s - f_k \le -\frac{\kappa}{2}\delta_k^2$ and $\delta_{k+1} = \gamma \delta_k$.
One thus has
\begin{align*}
\phi_{k+1} - \phi_k
&= \nu \left( f(x_{k+1}) - f(x_k) \right) + (1 - \nu)(\gamma^2 - 1)\delta_k^2 \\
&\le \left(\nu \left(2\varepsilon_f - \frac{\kappa}{2}\right) + (1 - \nu)(\gamma^2 - 1) \right) \delta_k^2 \\
&\le \left(- \nu \frac{\kappa}{4} + (1 - \nu)(\gamma^2 - 1) \right) \delta_k^2,
\end{align*}
where the last inequality is obtained using the fact that $\varepsilon_f \le \frac{\kappa}{8}$. Using~\eqref{eq:nu}, we have $\frac{\nu}{1-\nu} \ge \frac{8(\gamma^2 - 1)}{\kappa}$, which implies
\begin{align*}
\phi_{k+1} - \phi_k \le - \nu \frac{\kappa}{8} \delta_k^2 := b_2 < 0.
\end{align*}
We note also that ~\eqref{eq:nu} gives $\frac{\nu}{1 - \nu} \ge \frac{8}{\kappa}(1 - \beta_0^{2})$ and hence $b_2 \le b_1$. Hence, one deduces that for any iteration $k$, on this event, one has
\begin{align*}
\mathbf{1}_{T^f_k}\left(\Phi_{k+1} - \Phi_k\right) \le -\mathbf{1}_{T^f_k}(1 - \nu)(1- \beta_0^{2})\Delta_k^2 .
\end{align*}
Taking the conditional expectation with respect to $\mathcal{F}_{k-1}$ one gets
\[
\mathbb{E}\left(\mathbf{1}_{T^f_k} \left(\Phi_{k+1} - \Phi_k\right) \big| \mathcal{F}_{k-1}  \right)
\le - (1 - \nu)(1- \beta_0^{2} )p_k^*\Delta_k^2 \le - (1 - \nu)(1- \beta_0^{2} )p\,\Delta_k^2.
\]
\end{proof}
We now turn to the event where the estimates are inaccurate, on which the Lyapunov function may increase.

\begin{lemma} \label{lm:3}
Let Assumptions~\ref{ass:f} and~\ref{ass:ProbAccEstimates:obj} hold and let $\nu$ satisfy~\eqref{eq:nu}. Then, for every iteration $k$, one has
\[
\mathbb{E}\left(\mathbf{1}_{\overline{T^f_k}}(\Phi_{k+1} - \Phi_k) \big| \mathcal{F}_{k-1} \right)
\le 2 \nu (1-p_k^*)^{1/2} \varepsilon_v \Delta_k^2 \le 2 \nu (1-p)^{1/2} \varepsilon_v \Delta_k^2,
\]
where $\mathbf{1}_{\overline{T^f_k}}$ is the indicator function of the event set
$\overline{T^f_k}$, i.e., the indicator over the set of inaccurate estimates
$(F_k, F_k^{s})$.
\end{lemma}
\begin{proof}
Conditioned on $\mathcal{F}_{k-1}$, assume that the objective function estimates
$\{F_k, F_k^{s}\}$ are inaccurate. If the iteration $k$ is successful, one has $X_{k+1} = X_k + S_k$, $F_k^s - F_k \le -\frac{\kappa}{2}\Delta_k^2$ and $\Delta_{k+1} = \gamma \Delta_k$. Hence,
\begin{align*}
f(X_{k+1}) - f(X_k)
&= \left(f(X_k+S_k) - F_k^s\right) + \left(F_k^s - F_k\right) + \left(F_k - f(X_k)\right) \\
&\le \left|f(X_k+S_k) - F_k^s\right| + \left|F_k - f(X_k)\right| - \frac{\kappa}{2}\Delta_k^2,
\end{align*}
which leads to
\begin{align*}
\mathbf{1}_{\overline{T^f_k}}\mathbf{1}_{S}(\Phi_{k+1} - \Phi_k)
&\le \mathbf{1}_{\overline{T^f_k}}\mathbf{1}_{S} \Big( \nu \left|f(X_k+S_k) - F_k^s\right| + \nu \left|F_k - f(X_k)\right| \\
&\qquad\qquad\quad + \left(- \nu\frac{\kappa}{2} + (1-\nu)(\gamma^2-1)\right)\Delta_k^2 \Big).
\end{align*}
Using~\eqref{eq:nu}, we have $(1-\nu)(\gamma^2 - 1) \le \nu\frac{\kappa}{8} \le \nu\frac{\kappa}{2}$. Hence, one gets
\begin{equation} \label{eq:badsucc}
\mathbf{1}_{\overline{T^f_k}}\mathbf{1}_{S}(\Phi_{k+1} - \Phi_k)
\le \mathbf{1}_{\overline{T^f_k}}\mathbf{1}_{S}\, \nu \left( \left|f(X_k+S_k) - F_k^s\right| + \left|F_k - f(X_k)\right| \right).
\end{equation}
If the iteration is unsuccessful (certain, uncertain or empty), then $X_{k+1} = X_k$ and $\Delta_{k+1} \le \beta_0 \Delta_k$, which leads to
\begin{equation} \label{eq:badunsucc}
\mathbf{1}_{\overline{T^f_k}}\mathbf{1}_{\overline{S}}(\Phi_{k+1} - \Phi_k)
\le -\mathbf{1}_{\overline{T^f_k}}\mathbf{1}_{\overline{S}} (1 - \nu)(1-\beta_0^{2})\Delta_k^2 \le 0 .
\end{equation}
Since the right-hand side of~\eqref{eq:badsucc} is nonnegative, combining~\eqref{eq:badsucc} and~\eqref{eq:badunsucc} yields, in both cases,
\[
\mathbf{1}_{\overline{T^f_k}}(\Phi_{k+1} - \Phi_k)
\le \mathbf{1}_{\overline{T^f_k}}\, \nu \left( \left|f(X_k+S_k) - F_k^s\right| + \left|F_k - f(X_k)\right| \right).
\]
Taking the conditional expectation with respect to $\mathcal{F}_{k-1}$ and applying Lemma~\ref{lm:1}, one obtains
\begin{equation*}
\mathbb{E}\left(\mathbf{1}_{\overline{T^f_k}}(\Phi_{k+1} - \Phi_k) \big| \mathcal{F}_{k-1} \right)
\le 2 \nu (1-p_k^*)^{1/2} \varepsilon_v \Delta_k^2 \le 2 \nu (1-p)^{1/2} \varepsilon_v \Delta_k^2.
\end{equation*}
\end{proof}

Putting Lemmas~\ref{lm:2} and~\ref{lm:3} together yields the expected decrease of the Lyapunov function, provided that the estimates are accurate with a sufficiently high probability.
\begin{theorem} \label{thm:DecreaseOnPhi}
Let Assumptions~\ref{ass:f} and~\ref{ass:ProbAccEstimates:obj} hold, where the probability $p$ satisfies
\begin{equation} \label{eq:condp}
4\nu \varepsilon_v (1-p)^{1/2} \le (1-\nu)(1-\beta_0^{2})\, p ,
\end{equation}
and let $\nu \in (0, 1)$ satisfy~\eqref{eq:nu}. Then, there exists an $\alpha>0$ such that, for all
$k$,
\begin{equation} \label{eq:MainGoal}
\mathbb{E}\left(\Phi_{k+1} - \Phi_k \big| \mathcal{F}_{k-1} \right) \leq - \alpha \Delta_k^2.
\end{equation}
\end{theorem}

\begin{proof}
Putting the results of the two Lemmas~\ref{lm:2} and~\ref{lm:3} together, we obtain
\begin{align*}
\mathbb{E}\left(\Phi_{k+1} - \Phi_k \big| \mathcal{F}_{k-1} \right)
&= \mathbb{E}\left((\mathbf{1}_{T^f_k} + \mathbf{1}_{\overline{T^f_k}})(\Phi_{k+1} - \Phi_k) \big| \mathcal{F}_{k-1} \right) \\
&\leq \left( - (1 - \nu)p(1- \beta_0^{2}) + 2 \nu (1-p)^{1/2} \varepsilon_v \right)\Delta_k^2 .
\end{align*}
Hence, using~\eqref{eq:condp}, it reduces to
\[
\mathbb{E}\left(\Phi_{k+1} - \Phi_k \big| \mathcal{F}_{k-1} \right)
\leq - \frac{1}{2}(1 - \nu) p (1- \beta_0^{2}) \Delta_k^2 = - \alpha \Delta_k^2,
\]
where $\alpha = \frac{1}{2}(1 - \nu)p(1- \beta_0^{2}) > 0$.
\end{proof}
Hence, Theorem~\ref{thm:DecreaseOnPhi} ensures that the Lyapunov function decreases over iterations in expectation, as long as the accuracy probability of the estimates of $f$ is high enough. Using such a result, one can guarantee that the whole sequence of step sizes $\{\Delta_k\}$ converges to zero almost surely.
\begin{corollary} \label{corol:StepSizeToZeroAlmostSurely}
Suppose that the working assumptions of Theorem~\ref{thm:DecreaseOnPhi} are satisfied. Then, the sequence $\{\Delta_k\}_k$ goes to $0$ almost surely. Moreover, if the sequence of iterates $\{X_k\}$ is bounded, then there exist a subsequence $\mathcal{K}$ and a random vector $X_* \in \Omega$ such that $\{\Delta_k\}_{k \in \mathcal{K}}$ goes to zero and $\{X_k\}_{k \in \mathcal{K}}$ converges almost surely to $X_* \in \Omega$.
\end{corollary}

\begin{proof}
By taking expectation on the result from Theorem~\ref{thm:DecreaseOnPhi}, one gets
\begin{align*}
\mathbb{E}\left(\Phi_{n+1} - \Phi_0 \right)
&= \sum^{n}_{k=0} \mathbb{E}\left(\Phi_{k+1} - \Phi_k \right)
= \sum^{n}_{k=0} \mathbb{E}\left( \mathbb{E}\left(\Phi_{k+1} - \Phi_k \big| \mathcal{F}_{k-1} \right) \right) \\
&\le - \alpha \mathbb{E}\left( \sum^{n}_{k=0} \Delta_k^2 \right).
\end{align*}
Since $X_{n+1}\in\Omega$, one has $\Phi_{n+1} \geq \nu f_{\mathrm{low}}$ and by taking $n \to \infty$,
\begin{equation*}
\mathbb{E}\left( \sum^{+\infty}_{k=0} \Delta_k^2 \right) < \infty.
\end{equation*}
Thus, we conclude that the probability of the random variable $\Delta_k$ converging to zero is one. Moreover, assuming the boundedness of the sequence $\{X_k\}$, one deduces the existence of a random vector $X_*$ and a subsequence $\mathcal{K} \subset \mathbb{N}$ such that $\{\Delta_k\}_{k \in \mathcal{K}}$ goes to zero almost surely and
$\{X_k\}_{k \in \mathcal{K}} \subset \Omega$ converges almost surely to $X_* \in \Omega$ (as $\Omega$ is a closed set).
\end{proof}

Corollary~\ref{corol:StepSizeToZeroAlmostSurely} shows that the event
$A_\Delta=\{\lim_{k\to\infty}\Delta_k=0\}$ is almost sure. On this event, the
trust region eventually contains at least one point, so empty iterations can only occur finitely often. We now state this result formally.
\begin{proposition}\label{prop:TRnonempty}
Let the assumptions of Corollary~\ref{corol:StepSizeToZeroAlmostSurely} hold. Then, for every
$\omega\in A_\Delta$, one has $\overline{\Omega}_k\neq\emptyset$ for all $k$ sufficiently large. In particular, the set $\mathcal{U}_{\emptyset}$ of empty iterations is finite almost surely.
\end{proposition}

\begin{proof}
Let $\omega\in A_\Delta$. Since $\Omega$ has a nonempty interior,
$\operatorname{diam}(\Omega)>0$ and, by the triangle inequality,
$\sup_{y\in\Omega}\|y-x\|\ge\frac12\operatorname{diam}(\Omega)$ for every $x\in\Omega$.
As $\delta_k\to0$, there exists $k_1$ such that
$r_{\min}\delta_k<\frac12\operatorname{diam}(\Omega)$ for all $k\ge k_1$. For such $k$,
there exists $\bar y\in\Omega$ with $\|\bar y-x_k\|>r_{\min}\delta_k$. By convexity of
$\Omega$, the point
\[
y=x_k+\frac{r_{\min}\delta_k}{\|\bar y-x_k\|}\,(\bar y-x_k)
\]
lies on the segment $[x_k,\bar y]\subseteq\Omega$ and satisfies
$\|y-x_k\|=r_{\min}\delta_k\le r_{\max}\delta_k$. Hence $y\in\overline{\Omega}_k$ and
$\overline{\Omega}_k\neq\emptyset$ for all $k\ge k_1$. Consequently, $\mathcal{U}_{\emptyset}(\omega)\subseteq\{0,\ldots,k_1-1\}$.
\end{proof}

Before stating our main result, we recall an auxiliary result from the literature \cite{durrett2019probability} that will be useful for the analysis.

\begin{lemma} \label{submartingal:lm}
Let $\{G_k\}$ be a submartingale with respect to a filtration $\{\mathcal{F}_{k}\}$, i.e., $\mathbb{E}(G_k \big| \mathcal{F}_{k-1} ) \ge G_{k-1}$ for all $k$.  Assume further that there exists $M>0$ such that $|G_{k} -G_{k-1}| \le M <\infty$, for all $k$. Consider the random events
$C=\{\lim_{k\to \infty} G_k ~\mbox{exists and is finite}\}$ and $D=\{\lim \sup_{k\to \infty} G_k = \infty\}$. Then $\mathbb{P}(C\cup D)=1$.
\end{lemma}
For $k \in \mathcal{U}$, we define the random variable $\Psi_k$ with realizations
\[
\psi_k = \frac{f(x_k) - f(x_k + \delta_k q_k)}{\delta_k}, \qquad k \in \mathcal{U}.
\]

\begin{lemma} \label{lm:liminf}
Let the working assumptions of Corollary~\ref{corol:StepSizeToZeroAlmostSurely} hold. Then, almost surely, the set $\mathcal{U}$ of iterations where the local phase is performed is infinite and
\[
\liminf_{k \in \mathcal{U}} \Psi_k \le 0 .
\]
\end{lemma}

\begin{proof}
Let $Z_k=\sum^{k}_{i=0} \left(2\mathbf{1}_{T^f_i} -1\right)$, where $T^f_i$ is given in
Definition~\ref{defi:probaccfun} and recall that, by Assumption~\ref{ass:ProbAccEstimates:obj}, $p_i^*= \mathbb{P}(T^f_i \big| \mathcal{F}_{i-1} ) \ge p > \frac{1}{2}$. Since $Z_{k-1}$ is $\mathcal{F}_{k-1}$-measurable, $\{Z_k\}$ is a submartingale:
\[
\mathbb{E}(Z_k\big| \mathcal{F}_{k-1} ) = Z_{k-1} + 2 \mathbb{P}(T^f_k \big| \mathcal{F}_{k-1} ) -1 \ge Z_{k-1}.
\]
Note that $|Z_{k}- Z_{k-1}| =1$, hence the event $\{ \lim_{k\to \infty}Z_k ~\mbox{exists and is finite} \}$ is empty. Thus, by Lemma~\ref{submartingal:lm}, one deduces that the event $\{\lim \sup_{k} Z_k =\infty\}$ is almost sure. Recall also, from Corollary~\ref{corol:StepSizeToZeroAlmostSurely}, that the event $A_\Delta$ is almost sure.

Let $\epsilon>0$ be a rational number and consider the event
\[
E_{\epsilon} = \left\{ \omega \in \mathcal{E} : \exists k' \in \mathbb{N} ~\mbox{such that}~ \Psi_k(\omega) \ge \epsilon ~\mbox{for all}~ k \in \mathcal{U}(\omega),~ k \ge k' \right\}.
\]
We show that $E_\epsilon \cap A_\Delta \cap \{\lim \sup_{k} Z_k =\infty\} = \emptyset$. We proceed by contradiction and consider $\omega$ in this intersection, with the corresponding realizations $x_k=X_k(\omega)$, $\delta_k=\Delta_k(\omega)$, $q_k=Q_k(\omega)$,
$\psi_k=\Psi_k(\omega)$ and $z_k=Z_k(\omega)$. Since $\delta_k \to 0$ and by Proposition~\ref{prop:TRnonempty}, there exists $k_0 \ge k'$ such that, for all $k\ge k_0$, $k\notin\mathcal{U}_{\emptyset}$ and
$
\delta_k \le b_{\epsilon}:=\frac{2 \epsilon}{\kappa+4\varepsilon_f}$.

Let $B_k$ be the random variable whose realization is $b_k = \log_\gamma\left( \frac{\delta_k}{b_{\epsilon}}\right)$; thus, $b_k \le 0$ for all $k \ge k_0$. For any $k \ge k_0$, no empty iteration occurs and three cases remain.
\begin{itemize}
\item If $k \notin \mathcal{U}$, then, since $k\notin\mathcal{U}_{\emptyset}$, the global phase is successful, $\delta_{k+1} = \gamma\delta_k$ and $b_{k+1} = b_k + 1$.
\item If $k \in \mathcal{U}$ and $\mathbf{1}_{T^f_k} =1$, then $x^s_{k} = x^{\lcl}_{k} = x_k + \delta_k q_k$ and $f(x_k + \delta_k q_k) - f(x_k) = -\delta_k \psi_k \le -\epsilon\delta_k$. Hence, since  $\delta_k \le b_\epsilon$,
\[
f^{s}_k - f_{k} \le 2 \varepsilon_f \delta_k^2 - \epsilon \delta_k \le - \frac{\kappa}{2} \delta_k^2 .
\]
Hence, the iteration is successful, $\delta_{k+1}= \gamma \delta_k$ and $b_{k+1} = b_k + 1$.
\item If $k \in \mathcal{U}$ and $\mathbf{1}_{T^f_k} =0$, then, since
$\delta_{k+1} \ge \beta_2 \delta_k$ and $\gamma \beta_2 \ge 1$, one has $b_{k+1}\ge b_{k} + \log_\gamma (\beta_2)\ge b_{k} -1$.
\end{itemize}
In all cases, $b_{k+1} - b_k \ge 2\mathbf{1}_{T^f_k}(\omega) - 1 = z_k - z_{k-1}$ for all $k \ge k_0$. Summing these inequalities, one gets $b_k - b_{k_0} \ge z_{k-1} - z_{k_0-1}$ for all $k > k_0$. Since $\lim\sup_k z_k = \infty$, one obtains $\lim\sup_k b_k = \infty$, which contradicts the fact that $b_k \le 0$ for all $k \ge k_0$. Hence $E_\epsilon \cap A_\Delta \cap \{\lim \sup_{k} Z_k =\infty\} = \emptyset$ and since $A_\Delta$ and $\{\lim \sup_{k} Z_k =\infty\}$ are almost sure, $\mathbb{P}(E_\epsilon) = 0$.

Let $A_{\Psi} = \mathcal{E} \setminus \bigcup_{\epsilon \in \mathbb{Q}_{>0}} E_\epsilon$, which is almost sure as the complement of a countable union of null events. Consider $\omega \in A_\Psi$. If $\mathcal{U}(\omega)$ were finite, then, taking $k' > \max \mathcal{U}(\omega)$, one would have $\omega \in E_\epsilon$ for every $\epsilon$, which is impossible. Hence $\mathcal{U}(\omega)$ is infinite. Moreover, for every rational $\epsilon > 0$ and every $k'$, there exists $k \in \mathcal{U}(\omega)$, $k \ge k'$, such that $\psi_k < \epsilon$, i.e., $\liminf_{k\in\mathcal{U}(\omega)} \psi_k \le 0$.
\end{proof}
Lemma~\ref{lm:liminf} shows that, almost surely, the local phase is performed infinitely
often and fails to produce sufficient decrease along some subsequence of these iterations.
The analysis can therefore focus on the local phase iterations, and in particular on the
limit points they generate. This leads to the notions of refining subsequences and
refining directions associated with the local phase.

\begin{definition}\label{defi:refining}
Consider a realization of Algorithm~\ref{alg:framework}. A subsequence of iterates
$\{x_k\}_{k\in\mathcal{K}}$, with $\mathcal{K}\subseteq\mathcal{U}$ infinite, is called a
\emph{refining subsequence} if $\lim_{k\in\mathcal{K}}\delta_k=0$ and $\{x_k\}_{k\in\mathcal{K}}$
converges to some $x_*\in\Omega$, called a \emph{refined point}. A vector $q\in\mathbb{R}^d$
is a \emph{refining direction} for $x_*$, associated with $\mathcal{K}$, if there exists an
infinite subset $\mathcal{K}'\subseteq\mathcal{K}$ such that $\lim_{k\in\mathcal{K}'}q_k=q$.
\end{definition}

For $k\in\mathcal{K}\subseteq\mathcal{U}$, the local phase is performed at iteration
$k$, so that $x^s_k=x^{\lcl}_k\in\overline{\Omega}_k$. By definition of $\overline{\Omega}_k$ and of $q_k$,
one has $\|q_k\|=\|x^{\lcl}_k-x_k\|/\delta_k$ and hence
$r_{\min}\le\|q_k\|\le r_{\max}$.
Therefore, refining directions always exist
and since this annulus is closed, any refining direction $q$ satisfies $r_{\min} \le \|q\| \le r_{\max}$. 
Finally, since the iterates are bounded by assumption, refining
subsequences exist whenever $\mathcal{K}$ is infinite and
$\{\delta_k\}_{k\in\mathcal{K}}$ converges to zero.

Since we are in the constrained case, $y+tq$ is not guaranteed to belong to
$\Omega$ for $y$ close to $x_*$ and small $t>0$, so the classical Clarke
generalized derivative cannot be used. We now introduce the tools needed to
state our final convergence result for the constrained setting.

\begin{definition}
\label{defi:hypertangent}
A vector $q\in\mathbb{R}^d$ is a \emph{hypertangent vector} to $\Omega$ at $x\in\Omega$
if there exists $\varepsilon>0$ such that
\[
y+tu\in\Omega\qquad\forall\,y\in\Omega\cap B(x;\varepsilon),\ u\in B(q;\varepsilon),\ 0<t<\varepsilon .
\]
The set of such vectors, denoted by $T^{H}_{\Omega}(x)$, is an open cone called the
\emph{hypertangent cone}.
\end{definition}

When $f$ is Lipschitz continuous relative to $\Omega$ near $x\in\Omega$, the
Clarke--Jahn generalized derivative of $f$ at $x$ in the direction
$q\in T^{H}_{\Omega}(x)$ is defined by
\begin{equation}\label{eq:clarkejahn}
f^{\circ}(x;q):=\limsup_{\substack{y\to x,\ y\in\Omega\\ t\downarrow0,\ y+tq\in\Omega}}\frac{f(y+tq)-f(y)}{t}.
\end{equation}

We now state the main global convergence result for Algorithm~\ref{alg:framework}.

\begin{theorem} \label{th:12}
Let the working assumptions of Corollary~\ref{corol:StepSizeToZeroAlmostSurely} hold, and
assume also that the sequence of iterates $\{X_k\}$ generated by
Algorithm~\ref{alg:framework} is almost surely bounded. Then, there exists an almost sure event $A$ such that, for all $\omega \in A$, there exists a refining subsequence $\{X_k(\omega)\}_{k\in\mathcal{K}(\omega)}$ associated with the local phase, with refined point $X_*(\omega)\in\Omega$, such that $\lim_{k \in \mathcal{K}(\omega)} \Psi_k(\omega) \le 0$. Moreover, any refining direction $Q(\omega)$ for $X_*(\omega)$ satisfies $r_{\min} \le \|Q(\omega)\| \le r_{\max}$ and, if $Q(\omega)\in T^{H}_{\Omega}\big(X_*(\omega)\big)$,
\[
f^{\circ}\left(X_*(\omega); Q(\omega)\right) \ge 0 .
\]
\end{theorem}
\begin{proof}
From Corollary~\ref{corol:StepSizeToZeroAlmostSurely}, Lemma~\ref{lm:liminf} and the fact that
the sequence $\{X_k\}$ is almost surely bounded, the event
$A = A_\Delta \cap A_\Psi \cap \{\{X_k\}~\text{is bounded}\}$ is almost sure as the intersection
of three almost sure events. Now, consider $\omega \in A$ and let $x_k= X_k(\omega)$,
$\delta_k= \Delta_k(\omega)$, $q_k= Q_k(\omega)$ and $\psi_k = \Psi_k(\omega)$. By
Lemma~\ref{lm:liminf}, there exists an infinite subset $\mathcal{K}_0\subseteq\mathcal{U}$ such that
$\lim_{k \in \mathcal{K}_0} \psi_k = \liminf_{k\in\mathcal{U}}\psi_k \in[-\infty,0]$. Since
$\{x_k\}_{k\in\mathcal{K}_0}$ is bounded, there exists an infinite subset
$\mathcal{K}=\mathcal{K}(\omega)\subseteq\mathcal{K}_0$ such that $\{x_k\}_{k\in\mathcal{K}}$
converges to some $x_*$, which belongs to $\Omega$ since $\Omega$ is closed. As
$\omega\in A_\Delta$, one also has $\lim_{k\in\mathcal{K}}\delta_k=0$, so that
$\{x_k\}_{k\in\mathcal{K}}$ is a refining subsequence with refined point $x_*$, and
$\lim_{k\in\mathcal{K}}\psi_k=\lim_{k\in\mathcal{K}_0}\psi_k\le0$.

Since $f$ is locally Lipschitz continuous, there exist $r>0$ and $\lambda>0$ such that $f$ is
Lipschitz continuous with constant $\lambda$ on $\Omega\cap B(x_*;r)$. As $x_k\to x_*$,
$\delta_k\to0$ and $\|q_k\|\le r_{\max}$ for $k\in\mathcal{U}$, the points $x_k$ and
$x_k+\delta_kq_k=x^{\lcl}_k\in\overline{\Omega}_k\subseteq\Omega$ belong to $\Omega\cap B(x_*;r)$ for
$k\in\mathcal{K}$ sufficiently large. Hence $|\psi_k|\le\lambda\|q_k\|\le\lambda r_{\max}$ for such
$k$, and $\lim_{k\in\mathcal{K}}\psi_k$ is finite.

Since $\{q_k\}_{k\in\mathcal{K}}$ lies in the compact annulus
$\{q : r_{\min}\le\|q\|\le r_{\max}\}$, refining directions for $x_*$ associated with
$\mathcal{K}$ exist and all satisfy $r_{\min}\le\|q\|\le r_{\max}$. Let $q$ be any such
refining direction, and let $\mathcal{K}'\subseteq\mathcal{K}$ be an infinite subset such
that $\lim_{k\in\mathcal{K}'}q_k=q$.

Assume now that $q\in T^{H}_{\Omega}(x_*)$ and let $\varepsilon>0$ be given by
Definition~\ref{defi:hypertangent} for $q$ at $x_*$. Since $x_k \to x_*$ and
$\delta_k \to 0$, one has, for $k \in \mathcal{K}'$ sufficiently large,
$x_k\in\Omega\cap B(x_*;\varepsilon)$ and $0<\delta_k<\varepsilon$; applying
Definition~\ref{defi:hypertangent} with $y=x_k$, $u=q$ and $t=\delta_k$ gives
$x_k+\delta_kq\in\Omega$. Moreover, since $\|q\|\le r_{\max}$, one also has
$x_k+\delta_kq\in B(x_*;r)$ for $k\in\mathcal{K}'$ sufficiently large. Therefore, for
$k \in \mathcal{K}'$ sufficiently large, the three points $x_k$, $x_k + \delta_k q_k$ and
$x_k + \delta_k q$ belong to $\Omega\cap B(x_*;r)$ and, by Lipschitz continuity of $f$ on
$\Omega\cap B(x_*;r)$,
\begin{align*}
\frac{f(x_k + \delta_k q) - f(x_k)}{\delta_k}
&= \frac{f(x_k + \delta_k q) - f(x_k+\delta_kq_k)}{\delta_k} + \frac{f(x_k + \delta_k q_k) - f(x_k)}{\delta_k}
\ge -\lambda \|q - q_k\| -\psi_k .
\end{align*}
Since $y = x_k\in\Omega$, $y\to x_*$, $t = \delta_k \downarrow 0$ and
$y+tq=x_k+\delta_kq\in\Omega$ for $k \in \mathcal{K}'$ sufficiently large, these points are
admissible in~\eqref{eq:clarkejahn} and one gets
\begin{align*}
f^{\circ}(x_*; q)
&\ge \limsup_{k \in \mathcal{K}'} \frac{f(x_k + \delta_k q) - f(x_k)}{\delta_k}
\ge \limsup_{k \in \mathcal{K}'} \left( -\psi_k - \lambda \|q - q_k\| \right)
= - \lim_{k \in \mathcal{K}} \psi_k \ \ge\ 0,
\end{align*}
where the equality uses $\lim_{k\in\mathcal{K}'}\|q - q_k\| = 0$ and the fact that
$\{\psi_k\}_{k\in\mathcal{K}'}$ is a subsequence of the convergent sequence
$\{\psi_k\}_{k\in\mathcal{K}}$.
\end{proof}
If, in addition, the set of normalized refining directions is dense in
$T^{H}_{\Omega}(X_*(\omega))\cap\{u : \|u\|=1\}$, then $f^{\circ}(X_*(\omega);q)\ge0$ for all
$q\in T^{H}_{\Omega}(X_*(\omega))$, meaning $X_*(\omega)$ is a Clarke stationary point of $f$ over $\Omega$~\cite{AudetDennis2006}.
\section{STREGO: A scalable implementation of Algorithm~\ref{alg:framework}}
\label{sec:STREGO}
In this section, we introduce STREGO, a scalable instance of the generalized framework defined in Section~\ref{sec:Framework} tailored to noisy, high-dimensional blackbox problems. To this end, we design a batch-based global phase coupled with an exploitative local phase. We first describe the overall procedure, then detail the global and local phases. STREGO begins by sampling an initial set of points $\{x_i\}_{i=1}^{N_0}$ and evaluating their objective values. We denote the initial sample set by $\mathcal{X}=\{x_i\}_{i=1}^{N_0}$ and the corresponding evaluations by $\mathcal{Y}=\{f_i \approx f(x_i)\}_{i=1}^{N_0}$.  
Together, the pairs $\{(x_i,f_i)\}_{i=1}^{N_0}$ define the initial DoE $\mathcal{D}$. We also define a starting point $x_0$.
The algorithm then alternates between a global and a local phase. The global phase consists of training a GP on the entire DoE and solving a bi-objective problem based on the posterior information. Each global iteration produces a batch of $\mathrm{G}$ points, denoted $\mathcal{X}^{\mathrm{G}}$, which are evaluated. The global solution $x^{{\glb}}_{k}$ is the best point among the $\mathrm{G}$ candidates (where $k$ is the number of completed global-local steps).
If the global phase is unsuccessful (according to \eqref{eq:sufficient-decrease-global-framework}), the algorithm proceeds to the local search  within a trust region around the incumbent \eqref{eq:trust-region-TREGO}. In this phase, the local candidate $x_{k}^{\lcl}$ minimizes the posterior mean of a local surrogate model.
We now detail the global phase and the local phase.
\subsection{The global phase of STREGO}
\label{sec:global-phase}
The global phase starts by building and training a GP on the entire DoE. Two quantities are obtained from the model: the posterior mean $\mu$ and the posterior variance $\sigma^2$. The mean serves as an estimate of the objective function, while the variance represents the model's uncertainty about this estimate. Recall that EGO suggests the next candidate by the scalar acquisition function EI. However, as shown in~\cite{de2021greed}, EI tends to suggest points that have already been evaluated or that lie close to previously evaluated ones. Another difficulty appears in high dimension, where EI tends to propose points on the boundary of the search space. This is known as the boundary issue~\cite{siivola2018correcting,oh2018bock}. Indeed, $\sigma^2$ is a purely local measure. It quantifies the model's uncertainty at $x$ itself, independent of how that uncertainty relates to the rest of $\Omega$. This locality is the root of the boundary issue. Near the boundary of $\Omega$, the kernel's correlation neighborhood is truncated by the domain edge. Since no observations can lie on the other side, boundary points receive less information than interior points. This inflates their variance relative to points near the center and the effect worsens with dimension. Because variance-based acquisition criteria reward high $\sigma^2$ directly, they are repeatedly drawn toward these boundary regions, regardless of whether these regions are actually the most informative. Moreover, a scalar acquisition function captures only one exploration--exploitation trade-off. 

To address this, some works directly sample from the Pareto front of a bi-objective \cite{de2021greed,jiang2025trading}, or more generally multi-objective \cite{lyu2018batch}, problem built on the posterior of the surrogate. More specifically, Carciaghi \textit{et~al.}~\cite{carciaghi2025bi} formulate this as a mean--variance bi-objective problem, whose infill criterion is given by:
\begin{equation}
\min_{x \in \Omega} \, \big(\mu(x), -\sigma^2(x)\big).
\label{prob:biobj-prob-variance}
\end{equation}

Although separating the mean and variance gives access to diverse trade-offs, ranking candidates by $\sigma^2$ does not mitigate the boundary issue. In parallel, several works have developed stepwise uncertainty reduction (SUR) strategies for GP surrogates \cite{bect2012sequential,picheny2014stepwise}. At each step, the next candidate is chosen to minimize the expected remaining uncertainty about a quantity of interest once it is observed. Within this framework, taking the posterior variance as the uncertainty measure recovers the integrated mean squared error (IMSE) criterion of Sacks \textit{et~al.}~\cite{sacks1989design}, originally a one-shot design criterion for globally accurate surrogates. More recently, it was applied sequentially in active learning, where it is known as the integrated mean squared prediction error (IMSPE) \cite{binois2024gearing}. The same quantity appears in the literature as integrated variance reduction ($\mathrm{IVR}$) and as the negative integrated posterior variance up to an additive constant. Rather than ranking a candidate solely by its current predictive variance, $\mathrm{IVR}$ evaluates it according to the expected reduction in posterior variance over a reference domain. Therefore, it favors observations whose influence extends across a large portion of $\Omega$. 

Formally, given the current GP posterior with predictive
variance $\sigma^2$, the integrated variance reduction of a candidate $x$ is defined as: 
\begin{equation}
\mathrm{IVR}(x) = \int_\Omega \left[ \sigma_{\mathcal{D}}^2(x') - \sigma_{\mathcal{D}\cup\{x\}}^2(x') \right] \pi(x') \, dx',
\label{eq:ivr-def}
\end{equation}
where $\sigma^2_{\mathcal{D}}$ and $\sigma^2_{\mathcal{D}\cup\{x\}}$ denote the posterior
variance of the surrogate before and after adding $x$ to the current set of
observations $\mathcal{D}$, and $\pi$ is a reference measure over $\Omega$. Thus, IVR measures how much querying $x$ would reduce the uncertainty of the model over the entire domain $\Omega$, rather than the uncertainty at $x$ alone. This distinction is central in the expensive setting considered here. Under a fixed evaluation budget, each evaluation should improve the model where it matters for the search, not only at the candidate itself. By favoring candidates whose information propagates across the domain, IVR extracts more information from each evaluation and it does not reward boundary points merely because they are isolated. Note that \eqref{eq:ivr-def} depends on the data only through the design locations $\mathcal{X}$, not through the observed responses $\mathcal{Y}$. In practice we approximate $\pi$ by a fixed reference set and evaluate the resulting criterion in closed form. Further details are given in Section~\ref{sec:implementation_dets}.

In this context, our proposed bi-objective problem for the global phase is given by
\begin{equation}
    \min_{x \in \Omega} \, \big(\mu(x), -\mathrm{IVR}(x)\big).
    \label{prob:global-biobj-strego}
\end{equation}
Solving this bi-objective problem yields an approximation of the Pareto set $\mathcal{X}^{\mathrm{P}} \subseteq \Omega$. This set typically contains many more points
than the batch size $\mathrm{G}$, so only a subset of them can be evaluated. Choosing this
subset poorly may result in candidates that carry essentially the same information,
thereby wasting costly evaluations. Both
the predictive mean and the $\mathrm{IVR}$ criterion are derived from similarity
kernels defined in terms of input-space distances. Consequently, low-mean candidates
lie near previously well-performing observations, while high-$\mathrm{IVR}$
candidates lie in undersampled regions, and these regions are typically well
separated in the input space. Following this reasoning, and as done in~\cite{carciaghi2025bi},
we generate the batch of $\mathrm{G}$ global candidates by clustering the approximate Pareto
set $\mathcal{X}^{\mathrm{P}}$ in the input space.
Given a target batch size $\mathrm{G}$, we divide the Pareto set points into $\mathrm{G}$ clusters. 
A representative candidate is then selected from each cluster to ensure diversity. The resulting batch is set as
$
\mathcal{X}^{\mathrm{G}} = \{x_{i}^{\mathrm{G}}\}_{i=1}^{\mathrm{G}}$.
The DoE is augmented with the $\mathrm{G}$ evaluated solutions and the global solution $x^{\glb}_k$  is computed as the best point among them in terms of function value. 

Selecting $\mathrm{G}$ candidates at once allows them to be evaluated in parallel and gives a better coverage of the search space, which is difficult to achieve in high dimension. The main steps are reported in Algorithm~\ref{alg:strego_global_phase}. Given the corresponding input, it generates a batch of $\mathrm{G}$ points, evaluates them and returns both the augmented DoE and
the best point of the batch. In STREGO, that point is the global candidate $x^{\glb}_k$ used in Step 1 of Algorithm~\ref{alg:framework}.
\begin{algorithm}[!h]
    \footnotesize
    \caption{\bf Step 1 of Algorithm~\ref{alg:framework}.}
   \label{alg:strego_global_phase}
    \SetAlgoLined
    \LinesNotNumbered

    \KwData{Oracle  function $f$, feasible domain $\Omega$, batch size $\mathrm{G}\ge1$, DoE $\mathcal{D}=(\mathcal{X},\mathcal{Y})$.
    }
    \KwResult{$x^{\glb}$, $f^{\glb}$ and $\mathcal{D}$.
    }
\begin{algorithmic}[1]
\STATE Fit a GP surrogate on $\mathcal{D}$ with posterior mean $\mu$ and $\mathrm{\mathrm{IVR}}$.

\STATE Find a Pareto set $\mathcal{X}^{\mathrm{P}}$ for the bi-objective problem: $
\displaystyle \min_{x\in\Omega}\bigl(\mu(x),-\mathrm{IVR}(x)\bigr)$.

\STATE Select a batch $\mathcal{X}^{\mathrm{G}}=\{x_i\}_{i=1}^{\mathrm{G}}\subset \mathcal{X}^{\mathrm{P}}$ (e.g., via a clustering method).

\STATE Estimate the selected batch
$\mathcal{Y}^{\mathrm{G}}=\{f_i \approx f(x_i)\}_{i=1}^{\mathrm{G}}$.
\STATE Let $x^{\glb}$ be the best estimated point in the batch and $f^{\glb} \approx f(x^{\glb})$.
\STATE Update the DoE $
\mathcal{D}$, i.e.,
$
\mathcal{X} \leftarrow \mathcal{X}\cup \mathcal{X}^{\mathrm{G}} \text{ and }
\mathcal{Y} \leftarrow \mathcal{Y}\cup \mathcal{Y}^{\mathrm{G}}
$.
\end{algorithmic}
\end{algorithm}
\subsection{The local phase of STREGO}
\label{sec:local_phase}
The global phase explores the domain to generate promising starting points for local refinement. Because it must balance exploration and exploitation, the bi-objective formulation is well suited to this role. The local phase, by contrast, is purely exploitative. Once a promising region has been identified, the task is no longer to explore but to refine the estimate of the optimum within that region. Consequently, we discard the variance term and rely solely on the predictive mean inside the trust region:
\begin{equation}
\label{prob:min_local}
 \min_{x \in \overline{\Omega}} \, \mu(x),
\end{equation}
where $\overline{\Omega}$ is the local trust region around a given incumbent point $x_{\mathrm{inc}}$
\begin{equation}
    \overline{\Omega} = \left\{ x \in \Omega \;\middle|\; r_{\min}\,\delta \leq \|x - x_\mathrm{inc}\| \leq r_{\max}\,\delta \right\},
    \label{eq:trust-region-TREGO:inc}
\end{equation}
where $\delta$ is the trust-region radius and $0 < r_{\min} < r_{\max}$ are user-defined constants.

In TREGO's original implementation, the GP surrogate is trained on the entire DoE, which conflicts with the local nature of the trust-region logic, especially in high-dimensional settings. This choice is problematic for two reasons. First, fitting a single stationary kernel to the entire landscape forces one lengthscale to capture the behavior near the incumbent and the global structure of the objective. As the trust region shrinks, this mismatch can cause the surrogate to degenerate~\citep{eriksson2019scalable, tang2026rethinking}. Second, GP training  scales cubically with the number of points. Since the DoE grows by $\mathrm{G}$ points every global iteration, retraining on the full DoE becomes increasingly expensive. Therefore, we adopt an independent local model for the local phase and restrict its training set to the points lying within the trust region. This keeps the fitting cost bounded and allows the lengthscale to adapt to the local geometry of $\overline{\Omega}$ rather than to the whole domain. 
Problem~\eqref{prob:min_local} becomes
\begin{equation}
\label{prob:local-min-strego}
    \min_{x \in \overline{\Omega}} \, \overline{\mu}(x),
\end{equation}
where $\overline{\mu}$ is a local surrogate trained only using the local DoE. 
However, the trust region may contain too few evaluated points to train an accurate local model. We therefore train it on the relaxed local DoE
\begin{equation}
    \label{relaxed_TR}
     \mathcal{\overline{X}}
        =
      \left\{ x \in \mathcal{X} \;\middle|\;  \|x - x_{\text{inc}}\| \leq r_{\max}\,\rho \delta \right\},
\end{equation}
where $\rho\ge 1$ is a relaxation factor chosen so that the size of the local DoE $\mathcal{\overline{X}}$ is above a certain threshold $n_{\min}$. Relaxing the trust region enlarges the pool of already evaluated points available to the local model, providing more training data and improving its accuracy in the neighborhood of the incumbent. The local candidate $x^{\lcl}_k$ is then obtained by solving (\ref{prob:local-min-strego}), where the local GP is trained on the whole local DoE $\mathcal{\overline{X}}$. Given the current DoE, incumbent $x_k$ and $\delta_k$, Algorithm~\ref{alg:strego_local_phase} generates the local candidate $x^{\lcl}_k$ and updates the DoE with the new evaluation.
\begin{algorithm}[htbp] 
    \footnotesize
    \caption{\bf Step 3 of Algorithm~\ref{alg:framework}.}
   \label{alg:strego_local_phase}
    \SetAlgoLined
    \LinesNotNumbered

    \KwData{Oracle function $f$, incumbent $x_{\text{inc}}$, trust-region radius $\delta>0$, constants $r_{\max}> r_{\min}>0$, a size threshold $n_{\min}>1$ and the DoE $\mathcal{D}=(\mathcal{X},\mathcal{Y})$.
    }
    \KwResult{$x^{\lcl}$, $f^{\mathrm{\lcl}}$ and $\mathcal{D}$.
    }
\begin{algorithmic}[1]
\STATE Define the relaxed local DoE
\[
  \overline{\mathcal{X}} = \bigl\{x \in \mathcal{X} \;\big|\; \|x - x_k\| \le \rho\, r_{\max} \delta_k\bigr\},
\]
where $\rho \ge 1$ is the smallest value such that $|\overline{\mathcal{X}}| \ge n_{\min}$. Set $\overline{\mathcal{Y}} = \{f_i : x_i \in \overline{\mathcal{X}}\}$.
\STATE Fit a local GP on $\overline{\mathcal{D}} = (\overline{\mathcal{X}}, \overline{\mathcal{Y}})$ and let $\overline{\mu}$ denote its posterior mean.
\STATE Minimize the posterior mean $\overline{\mu}$ over the local trust region, as defined in ~\eqref{eq:trust-region-TREGO:inc}, i.e., 
$$
x^{\lcl} \in  \operatorname*{argmin}_{x \in \overline{\Omega}} \overline{\mu}(x)
\quad \text{and} \quad
f^{\mathrm{\lcl}} \approx f(x^{\lcl}).
$$
\STATE Update the DoE $\mathcal{D}$, i.e.,
$\mathcal{X} \gets \mathcal{X} \cup \{x^{\mathrm{\lcl}}\}$ and $
\mathcal{Y} \gets \mathcal{Y} \cup \{f^{\mathrm{\lcl}}\}$.
\end{algorithmic}
\end{algorithm}

\section{Numerical experiments}
\label{sec:numerical}
The experiments in this section serve two purposes. The first is to assess the choice of $\mathrm{IVR}$, rather than the predictive variance, as the uncertainty measure of the global phase. The second is to compare STREGO with TURBO, BAXUS, NSMA-BO and TREGO.
\subsection{Experimental setting}
We first use five synthetic benchmark functions, described in Appendix~\ref{app:test-prob} together with the noise models and the shift procedure. Two are separable (Alpine-01, Schwefel), two are nonseparable (Ackley, Rotated Rastrigin) and one, Hartmann-6, has a low-dimensional structure that a solver can exploit. Each function is tested in dimensions $d \in \{4, 20, 50, 100\}$, without noise and with two types of Gaussian noise. The additive noise has a standard deviation of approximately 10\% of the function's output range in the dimension considered, and the multiplicative noise a relative standard deviation of 0.1. Hartmann-6 requires $d \geq 6$ and is therefore not tested at $d = 4$. To avoid favoring certain solvers, we shift the functions whose optimum lies at the origin or near the boundary of the feasible domain. We then complete the comparison on real-world engineering problems from the blackbox benchmark Bencher~\cite{papenmeier2025bencher}.

\subsection{Implementation details}
\label{sec:implementation_dets}
STREGO was implemented in Python 3 and the experiments were run on a Linux machine with an Intel 12th Gen Core i7-12700 (x86\_64) CPU, using 6 CPU cores.

The GP of each phase is a \texttt{SingleTaskGP} model from BoTorch~\cite{balandat2020botorch}. It uses an unscaled radial basis function (RBF) kernel with automatic relevance determination (one lengthscale $\ell_i$ per coordinate) and the dimensionality-scaled lengthscale prior of Hvarfner \textit{et~al.}~\cite{hvarfner2024vanilla}.

In the global phase, we approximate IVR~\eqref{eq:ivr-def} by the average variance reduction over 64 reference points spread evenly across the domain by a Sobol sequence. These points are generated once and reused throughout, so that all candidates are compared on the same basis. The global bi-objective problem is solved with NSMA~\cite{lapucci2023memetic}, with a population of $100$ individuals and at most $20$ generations. Its operators are simulated binary crossover (SBX), with probability $0.9$ and distribution index $\eta_c = 20$, and polynomial mutation, with distribution index $\eta_m = 20$. A crowding-based mechanism, with a crowding quantile of $0.9$, promotes diversity along the Pareto front, and a gradient-based refinement step is performed every $n_{\mathrm{opt}} = 5$ iterations. The other settings of NSMA are its defaults.

We select the batch from the approximate Pareto set returned by NSMA, using $k$-means clustering in the input space with $k=\mathrm{G}$ and $n_{\text{init}}=10$ random initializations. Because the centroid of a cluster is the mean of its Pareto-set points and need not be one of them, we retain, for each cluster, the Pareto-set point nearest to its centroid. We set $\mathrm{G}=3$ because three clusters tend to yield one candidate with a low posterior mean, one with a high $\mathrm{IVR}$ and one compromise between the two. Appendix~\ref{app:batch_size_analysis} shows that $\mathrm{G}=3$ is preferable in high dimension, although $\mathrm{G}=1$ performs better at $d=4$. 

In the local phase, the local GP is trained on the relaxed local DoE~\eqref{relaxed_TR}, with $\rho \geq 1$ chosen as the smallest value such that $|\overline{\mathcal{X}}| \geq n_{\min} = 10$. Its posterior mean is minimized over the annulus~\eqref{eq:trust-region-TREGO:inc}, with $r_{\min}$ set to the machine precision, since Algorithm~\ref{alg:framework} requires $r_{\min} > 0$. The minimization uses the \texttt{optimize\_acqf} routine of BoTorch, with the posterior mean as acquisition function, $(2d+4)$ restarts and $(2d+4)^2$ raw samples drawn by Sobol sampling. For large $d$, the number of raw samples can be reduced to limit memory use.

The evaluation budget is limited to 200, which reflects the expensive setting considered in this paper. Although the candidates in a batch could be evaluated in parallel, we evaluate them sequentially. Each benchmark is run over $20$ independent trials, with all solvers started from the same initial DoE of size $N_0=10$. STREGO is run with the hyperparameters summarized in Table~\ref{tab:strego_hparams} of Appendix~\ref{app:hyperparameters} and each baseline with the default hyperparameters recommended in its original paper.

\subsection{Metrics}
\label{sec:metrics}
We use two metrics. The first, $f_{\mathrm{best}}$, is the best objective value found so far. Convergence curves show its mean over the $20$ trials as a function of the number of evaluations, with shaded regions indicating the 95\% confidence interval around the mean. The second, $f_{\mathrm{final}}$, is the best objective value a solver reaches within the budget on a given benchmark. To compare problems of different scales, we normalize $f_{\mathrm{final}}$ for each function and dimension,
\[
\text{Normalized } f_{\mathrm{final}} =
\frac{f_{\mathrm{final}} - f_{\min}}{f_{\max} - f_{\min}},
\]
where $f_{\min}$ and $f_{\max}$ are the best and worst values observed across all solvers and all trials on that function at that dimension. A normalized value of $0$ thus corresponds to the best value of the group, and a value of $1$ to the worst. In the noisy settings, all reported values are the noisy values observed by the optimizer rather than the underlying true objective, so that the reported performance reflects what an algorithm could actually measure in practice. We also report the CPU time needed to first reach $f_{\mathrm{final}}$, excluding the cost of the initial design.

The normalized values of $f_{\mathrm{final}}$ are displayed as violin plots, one per solver. A violin gathers one value per benchmark function and trial, at a fixed dimension and noise setting. Its width is a kernel density estimate of these values, so that the violin is wide where many runs concentrate and narrow elsewhere. The central bar marks the median, and the upper and lower bands mark the worst and best runs. A violin concentrated near $0$ therefore indicates a solver whose final values are close to the best of the group across functions and trials, whereas a bimodal shape indicates a solver whose behavior differs from one function to another.

\subsection{Variance versus IVR}
We compare two variants of STREGO that differ only in the uncertainty objective of the global bi-objective problem: the predictive variance, as in problem~\eqref{prob:biobj-prob-variance}, and $\mathrm{IVR}$, as in problem~\eqref{prob:global-biobj-strego}. We refer to them as ($\mu$, $-\sigma^2$) and ($\mu$, $-\mathrm{IVR}$), respectively. Both use the local phase described in Section~\ref{sec:local_phase}. Figure~\ref{fig:bi_obj_ablation} shows the evolution of $f_{\mathrm{best}}$ for both variants on the four functions tested at every dimension.

On Ackley, Alpine-01 and Rotated Rastrigin, $\mathrm{IVR}$ generally yields better final values than the predictive variance, for $d$ ranging from $4$ to $100$. For instance, on the noise-free Ackley function at $d=50$, the final value is $\approx 17.3$ with $\mathrm{IVR}$, as opposed to $\approx 19.6$ with the variance. On Rotated Rastrigin under additive noise at $d=100$, the values are $\approx 1000$ and $\approx 1120$, respectively. The gap is smaller on Ackley under multiplicative noise at $d=20$ and $d=100$, where the two variants end close to each other. Schwefel is the exception: the two criteria perform comparably, and the variance-based variant is slightly better in some cases, in particular under multiplicative noise at $d=20$ ($\approx 5100$ as opposed to $\approx 5350$).
\begin{figure}[h!]
    \centering
    \includegraphics[width=\linewidth]{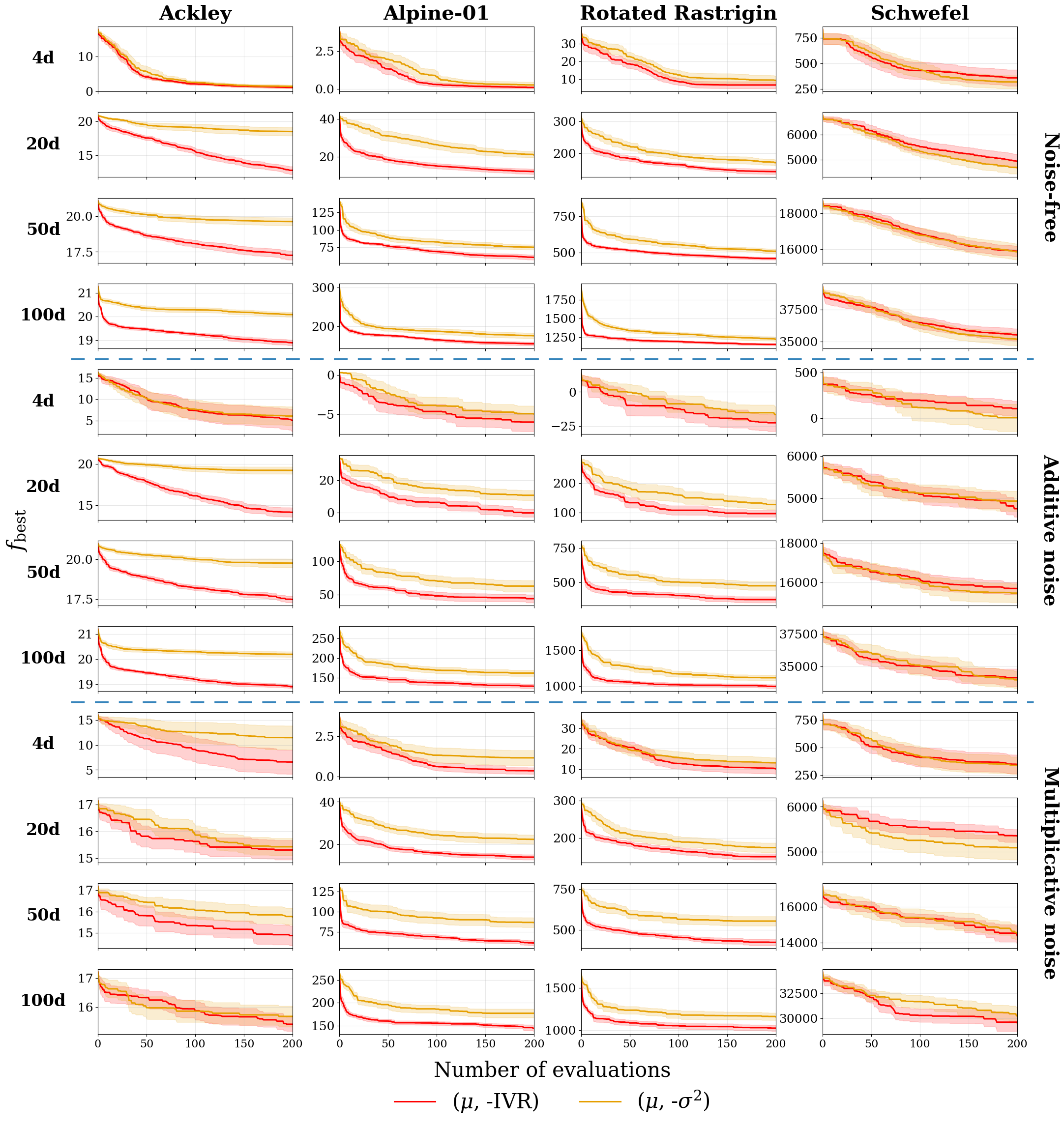}
    \caption{Comparison of STREGO with the global bi-objective criterion $(\mu, -\mathrm{IVR})$ against its variance counterpart $(\mu, -\sigma^2)$, both using the exploitative local phase, on Ackley, Alpine-01, Rotated Rastrigin and Schwefel in $4$d, $20$d, $50$d and $100$d under different noise settings.}
    \label{fig:bi_obj_ablation}
\end{figure}

\subsection{Comparison with state-of-the-art solvers}
We now compare STREGO with the four solvers named at the beginning of this section. TURBO~\cite{eriksson2019scalable} and BAXUS~\cite{papenmeier2022increasing} are well-established baselines for high-dimensional BO and handle the dimension in different ways. TURBO confines the search to one or several trust regions, runs an independent BO process in each and allocates evaluations among them through an implicit multi-armed bandit. BAXUS restricts the dimensionality of the search rather than its region: it optimizes over a nested sequence of random linear embeddings, expands the embedding whenever progress stalls and uses a TURBO-style trust region within each subspace. NSMA-BO~\cite{carciaghi2025bi} solves problem~\eqref{prob:biobj-prob-variance} of Section~\ref{sec:global-phase}, and the results reported in~\cite{carciaghi2025bi} show that it is also competitive in high dimension. Finally, TREGO~\cite{diouane2023trego}, described in Section~\ref{sec:trego}, is included because STREGO builds on it, so the comparison evaluates the effectiveness of the features that STREGO adds. We use the implementation of TREGO in egobox~\cite{Lafage2022}, within the Surrogate Modeling Toolbox (SMT)~\cite{saves2024smt}. This implementation uses the same GP in the global and local phases and EI as the acquisition function in both.

\subsubsection{Results on synthetic benchmarks}
We separate the effect of the dimension from that of the noise. Figure~\ref{fig:det_dim_analysis} varies the dimension in the noise-free setting, and Figures~\ref{fig:noise_analysis} and~\ref{fig:dim_noise_analysis} vary the noise setting at $d=4$ and $d=50$, respectively. Each figure shows the violins of the normalized $f_{\mathrm{final}}$, with tables of the mean normalized $f_{\mathrm{final}}$ and of the mean CPU time needed to first reach it.

\begin{figure*}[h!]
\centering


\includegraphics[width=\linewidth]{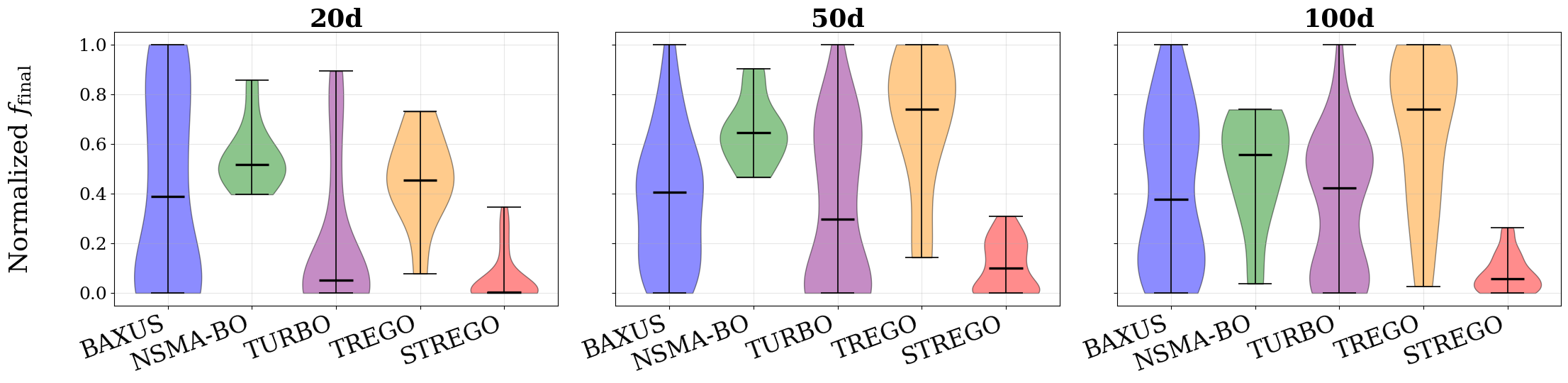}

\vspace{4mm}


\begin{minipage}[t]{0.49\textwidth}
\centering
\textbf{Normalized final objective value}\\[2mm]

\resizebox{\linewidth}{!}{
\begin{tabular}{lccc}
\toprule
Solver
& $20$d
& $50$d
& $100$d \\
\midrule

BAXUS
& $0.408 \pm 0.049$
& $0.376 \pm 0.038$
& $0.413 \pm 0.042$ \\

NSMA-BO
& $0.550 \pm 0.046$
& $0.652 \pm 0.041$
& $0.499 \pm 0.054$ \\

TURBO
& $0.200 \pm 0.041$
& $0.339 \pm 0.046$
& $0.354 \pm 0.036$ \\

TREGO
& $0.462 \pm 0.050$
& $0.694 \pm 0.040$
& $0.653 \pm 0.049$ \\

STREGO
& $\mathbf{0.061} \pm 0.018$
& $\mathbf{0.107} \pm 0.015$
& $\mathbf{0.080} \pm 0.010$ \\

\bottomrule
\end{tabular}
}
\end{minipage}
\hfill
%
\begin{minipage}[t]{0.49\textwidth}
\vspace{-9pt}
\centering
\textbf{CPU time (s)}\\[2mm]

\resizebox{\linewidth}{0.727\height}{%
\begin{tabular}{lccc}
\toprule
Solver
& $20$d
& $50$d
& $100$d \\
\midrule

BAXUS
& $253.23 \pm 13.44$
& $609.87 \pm 27.65$
& $689.12 \pm 55.23$ \\

NSMA-BO
& $475.95 \pm 3.23$
& $\mathbf{372.56} \pm 11.50$
& $\mathbf{503.46} \pm 14.18$ \\

TURBO
& $\mathbf{63.78} \pm 2.92$
& $515.66 \pm 24.66$
& $559.66 \pm 18.58$ \\

TREGO
& $176.83 \pm 8.50$
& $913.53 \pm 9.25$
& $6271.72 \pm 88.22$ \\

STREGO
& $334.48 \pm 25.43$
& $449.19 \pm 32.99$
& $518.80 \pm 29.79$ \\

\bottomrule
\end{tabular}
}
\end{minipage}

\caption{Normalized final objective value (lower is better) and CPU time for $d=20$, 50 and 100 in the noise-free setting. Tables report the mean $\pm$ standard error over 20 trials and 5 functions.}
\label{fig:det_dim_analysis}
\end{figure*}
STREGO obtains the best mean normalized $f_{\mathrm{final}}$ in every configuration of Figures~\ref{fig:det_dim_analysis} to~\ref{fig:dim_noise_analysis}, between 0.061 and 0.195, and the most compact violins. Its lead over the next best solver is largest under additive noise at $d=4$, where the other four solvers obtain means between 0.474 and 0.511, as opposed to 0.149 for STREGO.

The ranking of the other solvers depends on the dimension and on the noise. In the noise-free setting, TURBO ranks second at every dimension, with means between 0.136 and 0.354, and BAXUS ranks third for $d \geq 20$ (0.376 to 0.413). TREGO degrades as the dimension grows. It ranks third at $d=4$ (0.165), fourth at $d=20$ (0.462) and last at $d=50$ and $d=100$ (0.694 and 0.653), where its runs cover most of $[0, 1]$. This degradation is consistent with the two limitations of TREGO discussed in Sections~\ref{sec:global-phase} and~\ref{sec:local_phase}, namely the attraction of EI toward the boundary in high dimension and the single GP trained on the entire DoE. At $d=4$, TURBO remains second under multiplicative noise (0.212), whereas under additive noise the other four solvers lie within 0.04 of each other. At $d=50$, NSMA-BO becomes second under both types of noise (0.383 and 0.358), and TREGO remains last (0.584 and 0.454). TURBO falls to third under additive noise, level with BAXUS (0.458 and 0.463), and to fourth under multiplicative noise (0.436).

In terms of cost, STREGO is expensive for $d \leq 20$ but not for $d \geq 50$. At $d=4$, it is the second most expensive solver (280 to 347\,s), after NSMA-BO (389 to 452\,s), whereas TURBO, TREGO and BAXUS need 12 to 77\,s. A plausible cause is that both STREGO and NSMA-BO solve a bi-objective problem with NSMA at every global iteration, which costs more than the single-objective acquisition step of the other solvers. As the dimension grows, the cost of STREGO increases moderately, from 334\,s at $d=20$ to 519\,s at $d=100$, whereas that of TREGO rises from 177\,s to 6272\,s. For $d \geq 50$ in the noise-free setting, only NSMA-BO is cheaper than STREGO (373 and 503\,s, as opposed to 449 and 519\,s). Under noise at $d=50$, STREGO becomes the cheapest solver (356 and 385\,s), while the cost of TREGO increases to 1383 and 1271\,s.

\begin{figure*}[t]
\centering


\includegraphics[width=\linewidth]{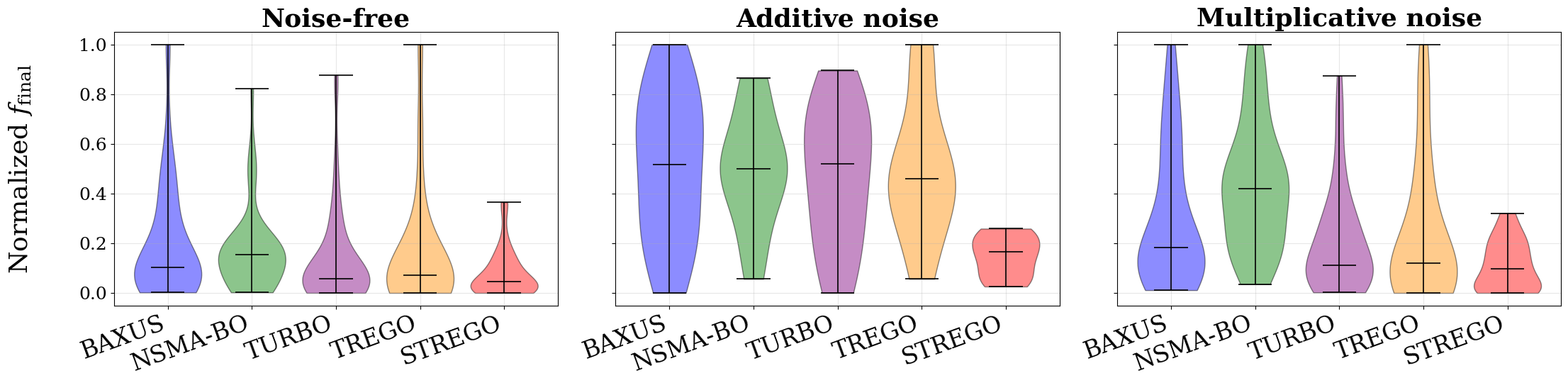}

\vspace{4mm}


\begin{minipage}[t]{0.49\textwidth}
\centering
\textbf{Normalized final objective value}\\[2mm]

\resizebox{\linewidth}{!}{
\begin{tabular}{lccc}
\toprule
Solver
& Noise-free
& Additive
& Multiplicative \\
\midrule

BAXUS
& $0.202 \pm 0.032$
& $0.511 \pm 0.054$
& $0.301 \pm 0.044$ \\

NSMA-BO
& $0.171 \pm 0.026$
& $0.494 \pm 0.048$
& $0.436 \pm 0.041$ \\

TURBO
& $0.136 \pm 0.026$
& $0.474 \pm 0.050$
& $0.212 \pm 0.034$ \\

TREGO
& $0.165 \pm 0.038$
& $0.499 \pm 0.045$
& $0.252 \pm 0.047$ \\

STREGO
& $\mathbf{0.080} \pm 0.013$
& $\mathbf{0.149} \pm 0.023$
& $\mathbf{0.111} \pm 0.016$ \\

\bottomrule
\end{tabular}
}
\end{minipage}
\hfill
\begin{minipage}[t]{0.49\textwidth}
\centering
\textbf{CPU time (s)}\\[2mm]

\resizebox{\linewidth}{!}{
\begin{tabular}{lccc}
\toprule
Solver
& Noise-free
& Additive
& Multiplicative \\
\midrule

BAXUS
& $76.98 \pm 2.06$
& $64.30 \pm 5.28$
& $60.76 \pm 2.83$ \\

NSMA-BO
& $452.45 \pm 14.37$
& $409.16 \pm 15.38$
& $389.02 \pm 11.89$ \\

TURBO
& $\mathbf{13.50} \pm 0.36$
& $\mathbf{11.94} \pm 0.62$
& $\mathbf{13.15} \pm 0.57$ \\

TREGO
& $35.00 \pm 1.42$
& $29.88 \pm 1.45$
& $27.89 \pm 1.09$ \\

STREGO
& $347.29 \pm 10.56$
& $298.41 \pm 27.44$
& $279.63 \pm 9.73$ \\

\bottomrule
\end{tabular}
}
\end{minipage}
\caption{Normalized final objective value (lower is better) and CPU time across different noise settings for $d=4$. Tables report the mean $\pm$ standard error over 20 trials and 4 functions.}
\label{fig:noise_analysis}
\end{figure*}
\begin{figure*}[h!]
\centering


\includegraphics[width=\linewidth]{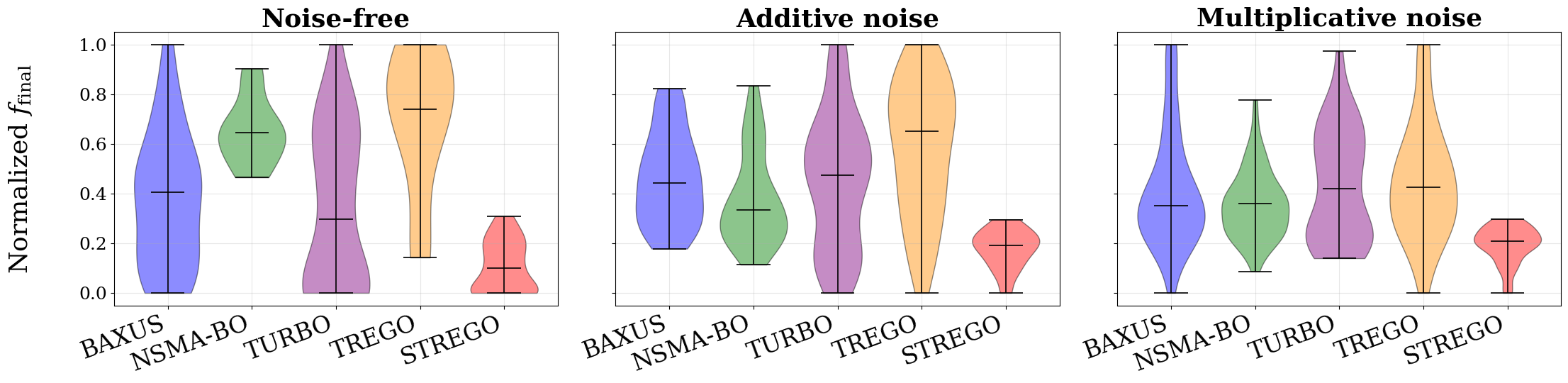}

\vspace{4mm}


\begin{minipage}[t]{0.49\textwidth}
\centering
\textbf{Normalized final objective value}\\[2mm]

\resizebox{\linewidth}{!}{
\begin{tabular}{lccc}
\toprule
Solver
& Noise-free
& Additive
& Multiplicative \\
\midrule

BAXUS
& $0.376 \pm 0.038$
& $0.463 \pm 0.025$
& $0.419 \pm 0.031$ \\

NSMA-BO
& $0.652 \pm 0.041$
& $0.383 \pm 0.026$
& $0.358 \pm 0.018$ \\

TURBO
& $0.339 \pm 0.046$
& $0.458 \pm 0.037$
& $0.436 \pm 0.030$ \\

TREGO
& $0.694 \pm 0.040$
& $0.584 \pm 0.037$
& $0.454 \pm 0.031$ \\

STREGO
& $\mathbf{0.107} \pm 0.015$
& $\mathbf{0.178} \pm 0.010$
& $\mathbf{0.195} \pm 0.010$ \\

\bottomrule
\end{tabular}
}
\end{minipage}
\hfill
\begin{minipage}[t]{0.49\textwidth}
\centering
\textbf{CPU time (s)}\\[2mm]

\resizebox{\linewidth}{!}{
\begin{tabular}{lccc}
\toprule
Solver
& Noise-free
& Additive
& Multiplicative \\
\midrule

BAXUS
& $609.87 \pm 27.65$
& $615.45 \pm 38.17$
& $591.74 \pm 32.77$ \\

NSMA-BO
& $\mathbf{372.56} \pm 11.50$
& $473.16 \pm 39.20$
& $494.36 \pm 31.38$ \\

TURBO
& $515.66 \pm 24.66$
& $550.57 \pm 24.11$
& $561.89 \pm 23.05$ \\

TREGO
& $913.53 \pm 9.25$
& $1382.85 \pm 85.35$
& $1270.90 \pm 74.19$ \\

STREGO
& $449.19 \pm 32.99$
& $\mathbf{356.40} \pm 21.30$
& $\mathbf{384.79} \pm 23.96$ \\

\bottomrule
\end{tabular}
}
\end{minipage}

\caption{Normalized final objective value (lower is better) and CPU time across different noise settings for $d=50$. Tables report the mean $\pm$ standard error over 20 trials and 5 functions.}
\label{fig:dim_noise_analysis}

\end{figure*}
\subsubsection{Results on real-world and engineering benchmarks}
The six real-world problems are Rover ($60$d), Mopta08 ($124$d) and four MuJoCo locomotion tasks: Swimmer ($16$d), Hopper ($33$d), Walker ($102$d) and HalfCheetah ($102$d). We add no artificial noise to them. Rover and Mopta08 are deterministic, whereas the MuJoCo tasks are inherently noisy, because each evaluation returns the cumulative reward of a policy whose initial state is randomly perturbed by the simulator. Since STREGO does not handle the blackbox constraints of Mopta08, we add them to the objective as penalty terms, as is common in high-dimensional BO~\cite{eriksson2019scalable,papenmeier2022increasing}. Figure~\ref{fig:realworldprob} shows the evolution of $f_{\mathrm{best}}$, the only metric reported for these problems.

STREGO outperforms all other solvers on Walker and Mopta08, where it takes the lead early and ends at $\approx -800$ and $\approx 262$, respectively, as opposed to $\approx -490$ for BAXUS and $\approx 275$ for TURBO, the next best solvers. On HalfCheetah, it ends slightly lower than BAXUS and far ahead of the other solvers. On Swimmer, all solvers end close to each other, with STREGO and TREGO slightly ahead. The remaining two problems favor other solvers. On Hopper, TURBO and BAXUS end lower than STREGO ($\approx -1700$ and $\approx -1300$, as opposed to $\approx -1100$), and on Rover, TREGO improves much faster at first and ends slightly lower ($\approx 6.3$ as opposed to $\approx 6.5$).

The relative performance of the other four solvers varies from one problem to another, which suggests that the problems differ in structure. The good results of BAXUS on Walker and HalfCheetah point to an exploitable low-dimensional structure in these problems. BAXUS is weaker on Swimmer and Rover, whose active subspace may be too high-dimensional for its random embeddings to capture within the budget. On Mopta08, which is known to have a high effective dimension~\cite{papenmeier2022increasing}, it ranks only third. TREGO ranks third on HalfCheetah, fourth on Mopta08, second to last on Walker and last on Hopper, where it barely improves after about $100$ evaluations. A plausible cause is that TREGO trains a single GP on the entire DoE, whose lengthscales may not fit the neighborhood of the incumbent once the trust region shrinks (Section~\ref{sec:local_phase}). NSMA-BO stalls early on Walker and Mopta08, and its progress on Mopta08 is essentially flat after about $50$ evaluations. This behavior may come from its variance objective, which favors isolated points near the boundary, where little is learned in high dimension (Section~\ref{sec:global-phase}). STREGO is designed to avoid both limitations, with $\mathrm{IVR}$ in the global phase and a local GP in the local phase, and it imposes no restriction on the effective dimension. We believe that these properties explain, at least in part, why it performs well both on Walker and HalfCheetah, where BAXUS is also strong, and on Mopta08, which has a high effective dimension.
\begin{figure}[h!]
    \centering
    \includegraphics[width=\linewidth]{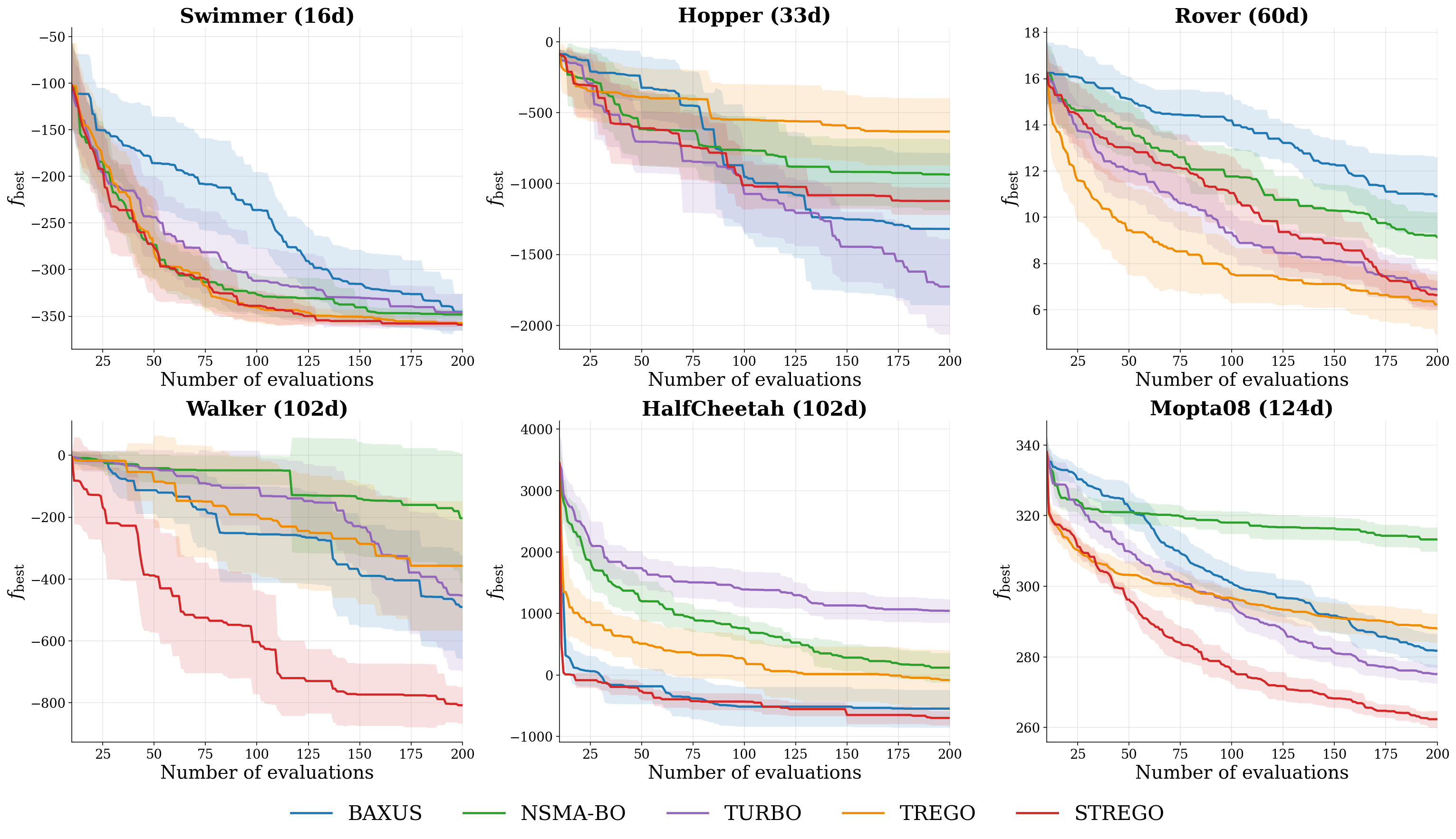}
    \caption{Convergence of STREGO, TREGO, TURBO, BAXUS and NSMA-BO on six real-world benchmarks of
increasing dimensionality ($16$d to $124$d).}
    \label{fig:realworldprob}
\end{figure}


\section{Conclusion}
\label{sec:conclusion}
At the outset, we asked whether a trust-region BO method could keep its convergence guarantees with noisy evaluations and remain competitive in high dimension. Our results indicate that it can. On the theoretical side, our framework extends the convergence analysis of TREGO to noisy function estimates that are sufficiently accurate with sufficiently high probability. On the practical side, the performance of STREGO seems largely due to two design choices. Its global phase ranks candidates by the uncertainty they remove over the whole domain rather than by their own uncertainty, and its local phase fits its model only to the points near the incumbent.

In our experiments, replacing the predictive variance by the integrated variance reduction generally improved the final values, except on Schwefel. STREGO gave the best solutions on average in every synthetic configuration tested, from dimension $4$ to $100$, with and without noise, whereas TREGO ranked last from dimension $50$ on. On the real-world problems, STREGO found the best solution on Walker, Mopta08 and HalfCheetah and tied with TREGO on Swimmer, while TURBO and BAXUS did better on Hopper and TREGO on Rover. Unlike BAXUS, which searches low-dimensional subspaces, STREGO searches the full space, and it performed well both on Walker and HalfCheetah, where BAXUS is strong, and on Mopta08, where BAXUS ranks third. Its weakness is its computational cost at low dimension. Up to dimension $20$, it was more expensive than three of the four other solvers, whereas from dimension $50$ on, only NSMA-BO was cheaper, and only without noise.

Future research will be devoted to three extensions of STREGO. The first is a cheaper way to solve the bi-objective problem of the global phase, to reduce the cost of STREGO at low dimension. The second is the handling of categorical and integer variables, which are common in the tuning of machine learning models, a setting where evaluations are noisy. The third is the direct handling of blackbox constraints and of simulations of varying fidelity, both common in aircraft design~\cite{priem2020efficient}. Because the framework leaves the estimation procedure open, we will also examine whether low-fidelity simulations can provide sufficiently accurate estimates.

\bibliographystyle{abbrv}
\bibliography{ref}

\appendix
\section{Synthetic test problems}\label{app:test-prob}

\paragraph{COCO/BBOB-inspired synthetic functions.}
Our first benchmark set is inspired by the COmparing Continuous Optimizers (COCO) platform \cite{hansen2016coco} and in particular its BBOB test suite.
We consider five standard analytical functions: Ackley, Alpine-01, Rotated Rastrigin, Schwefel and Hartmann-6. For each function, we report results in dimensions $d\in \{4,20,50,100\}$ except for Hartmann-6 which only allows $d\geq6$.

\paragraph{Ackley function.}
The Ackley function is a widely used multimodal benchmark, characterized by a nearly flat outer region and a deep, narrow basin around the global optimum.
\[
f_{\mathrm{Ackley}}(x)
=
-20\exp\!\left(-0.2\sqrt{\frac{1}{d}\sum_{i=1}^{d} x_i^2}\right)
-\exp\!\left(\frac{1}{d}\sum_{i=1}^{d}\cos(2\pi x_i)\right)
+ 20 + e.
\]
The global minimum is $f(x^\star)=0$ at $x^\star=(0,\ldots,0)$, for $x \in [-32.768,32.768]^d$.

\paragraph{Alpine function (Alpine-01).}
The Alpine-01 function is a nonsmooth benchmark with sharp ridges, commonly
used to assess robustness to nondifferentiability.
\[
f_{\mathrm{Alpine}}(x)
=
\sum_{i=1}^{d}
\left|
x_i \sin(x_i) + 0.1 x_i
\right|.
\]
The global minimum is $f(x^\star)=0$ at $x^\star=(0,\ldots,0)$, for
$x \in [-10,10]^d$.

\paragraph{Schwefel function.}
The Schwefel function is a benchmark whose global optimum lies geometrically distant from the next-best local minima.
\[
f_{\mathrm{Schwefel}}(x)
=
418.9829\,d
-
\sum_{i=1}^{d}
x_i \sin\!\left(\sqrt{|x_i|}\right).
\]
The global minimum is $f(x^\star)=0$ at $x^\star=(420.9687,\ldots,420.9687)$, for $x \in [-500,500]^d$.
\paragraph{Hartmann-6 function.}
The six-dimensional Hartmann function is a multimodal benchmark with six local minima and a narrow basin around the global optimum.
\[
f_{\mathrm{Hartmann6}}(x)
=
-\sum_{i=1}^{4}
\alpha_i
\exp\!\left(
-\sum_{j=1}^{6} A_{ij}\,\bigl(x_j - P_{ij}\bigr)^2
\right),
\]
where $\alpha = (1.0,\,1.2,\,3.0,\,3.2)^\top$,
\[
A =
\begin{pmatrix}
10 & 3 & 17 & 3.5 & 1.7 & 8\\
0.05 & 10 & 17 & 0.1 & 8 & 14\\
3 & 3.5 & 1.7 & 10 & 17 & 8\\
17 & 8 & 0.05 & 10 & 0.1 & 14
\end{pmatrix},
\qquad
P = 10^{-4}
\begin{pmatrix}
1312 & 1696 & 5569 & 124 & 8283 & 5886\\
2329 & 4135 & 8307 & 3736 & 1004 & 9991\\
2348 & 1451 & 3522 & 2883 & 3047 & 6650\\
4047 & 8828 & 8732 & 5743 & 1091 & 381
\end{pmatrix}.
\]
The global minimum is $f(x^\star)\approx -3.32237$ at $x^\star \approx (0.20169,\,0.15001,\,0.47687,\,0.27533,\,0.31165,\,0.65730)$, for $x \in [0,1]^6$.
We embed the standard 6-dimensional Hartmann function in higher dimensions ($d > 6$)  by evaluating it on six coordinates chosen at random (fixed for each d) and leaving the other coordinates as inactive dummy variables that do not affect the objective.
\paragraph{Rotated Rastrigin function.}
The rotated Rastrigin function is a highly multimodal benchmark with a regular grid of local minima. An orthogonal rotation couples the coordinates, so the problem cannot be split into one-dimensional subproblems.
\[
f_{\mathrm{RotRastrigin}}(x)
=
10\,d
+
\sum_{i=1}^{d}
\left[
z_i^2 - 10\cos\!\left(2\pi z_i\right)
\right],
\qquad
z = R\,x,
\]
where $R \in \mathbb{R}^{d \times d}$ is a fixed random orthogonal matrix.
The global minimum is $f(x^\star)=0$ at $x^\star=(0,\ldots,0)$, for $x \in [-5.12,5.12]^d$.
\paragraph{Noisy variants.}
To evaluate robustness under stochastic evaluations, we construct noisy variants of each synthetic function using two noise settings commonly used in the noisy blackbox optimization literature \cite{uchida2024cma}.

\textbf{\emph{Additive Gaussian noise}} corrupts the objective with noise that is calibrated to each function's output scale:
$$
\tilde f(x) = f(x) + \sigma_{f,d} \epsilon, \epsilon \sim \mathcal{N}(0,1),
$$
where $\sigma_{f,d}$ is the function and dimension specific noise standard deviation, calibrated to  approximately 10\% of the output range for function $f$ in dimension $d$. We estimate the output range of each objective via Monte Carlo sampling. For each $(f, d)$ pair, we uniformly sample 10,000 random points from the domain, evaluate $f$ at each point and compute the observed range $R_{f,d} = \max_i f(x_i) - \min_i f(x_i)$. We then set $\sigma_{f,d} = 0.10 \times R_{f,d}$. 

\textbf{\emph{Multiplicative Gaussian noise}} instead scales the perturbation with the function value:
\[
\tilde f(x) = f(x)\big(1 + \sigma \epsilon\big) ,
 \epsilon \sim \mathcal{N}(0,1),\]
where $\sigma=0.1$ .
\paragraph{Shifting.}
Ackley, Alpine-01 and Rotated Rastrigin have one of their global optima at the origin $x^\star = (0,\ldots,0)$, the exact center of the search domain. Evaluating them unshifted would favor solvers whose sampling is biased toward the center. Following the idea of~\cite{yu2025git}, we therefore evaluate shifted variants in which every coordinate is displaced:
\[
f_{\mathrm{shifted}}(x) = f(x - s), \qquad s_i \sim \mathcal{U}\big(x_{\mathrm{LB}} + 0.2\,\Delta,\; x_{\mathrm{UB}} - 0.2\,\Delta\big), \quad i=1,\ldots,d,
\]
where $\Delta = x_{\mathrm{UB}} - x_{\mathrm{LB}}$ is the width of the domain. The shifted optimum $x^\star = s$ thus lies uniformly in the central $60\%$ of the domain along each coordinate. All three functions are nonnegative on $\mathbb{R}^d$ and attain their minimum value $0$ at the origin, so the region of $x - s$ that falls outside the canonical domain cannot contain a better minimum.

Schwefel has the opposite problem. Its optimum $x^\star_i = 420.9687$ lies at $92.1\%$ of the domain $[-500, 500]$ in every coordinate, close to the upper bound. This placement favors solvers biased toward the boundary and penalizes solvers whose acquisition criteria integrate over the domain. Unlike the other three functions, Schwefel is not bounded below outside its canonical domain, so a large shift would let $x - s$ reach regions whose values compete with, or fall below, the true optimum. We therefore draw a negative shift from a narrow range:
\[
f_{\mathrm{shifted}}(x) = f(x - s), \qquad s_i \sim \mathcal{U}(-140, -125), \quad i=1,\ldots,d .
\]
This moves the optimum inward to $x^\star_i = 420.9687 + s_i \in [281.0, 296.0]$, that is, to $78$--$80\%$ of the domain, while keeping $f(x^\star) = 0$. The shift is capped at $140$ in magnitude so that the region of $x_i - s_i$ outside the canonical domain, $(500, 640]$, contributes a per-coordinate value of at least $313.6$. That is well above the second-best local minimum inside the original domain ($118.4$). The global optimum therefore stays unique. Moreover, since $x_i - s_i \in [-375, 640]$, all seven per-coordinate local minima of the original function remain inside the search domain, so the function is just as multimodal as before.

For every function, the shift vector $s$ is drawn once for each dimension $d$ from a fixed seed and the same $s$ is used in all trials and for all solvers.

\section{Hyperparameters}\label{app:hyperparameters}
All experiments use the same STREGO configuration across benchmarks, dimensions
and noise regimes. The settings are summarized in Table~\ref{tab:strego_hparams}.
\begin{table}[ht!]
\centering
\begin{tabular}{@{}l r@{}}
\toprule
\textbf{Hyperparameter} & \textbf{Setting} \\
\midrule
Initial DoE size ($N_0$)                          & $10$  \\
Evaluation budget             & $200$ \\
Number of trials                                  & $20$ \\
Base seed                                         & $42$ \\
Initial trust-region radius ($\delta_0$)          & $0.5\,(1/5)^{1/d}$ \\
Expansion factor ($\gamma$)                       & $2.0$ \\
Contraction factor, empty iteration ($\beta_0$) & $0.7$ \\
Contraction factor, uncertain unsuccessful ($\beta_1$) & $0.7$ \\
Contraction factor, certain unsuccessful ($\beta_2$)   & $0.5$ \\
$\kappa$                   & $1.0$ \\
Minimum number of local points ($n_{\min}$)       & $10$ \\
Global batch size ($\mathrm{G}$)       & $3$ \\
$r_{\max}$      & $1$ \\
$r_{\min}$   & $\epsilon_m \approx 2.2 \times 10^{-16}$
\end{tabular}
\vspace{4mm}
\caption{STREGO hyperparameters used for a $d$-dimensional problem.}
\label{tab:strego_hparams}
\end{table}

\FloatBarrier
\section{Sensitivity analysis: global batch size}
\label{app:batch_size_analysis}
Figure~\ref{fig:batchsize_ablation} compares STREGO with the global batch sizes $\mathrm{G}\in \{1, 2, 3, 4\}$, with all other settings and the evaluation budget held fixed. At $d=4$, $\mathrm{G}=1$ performs best on all four functions and in all three noise settings, except on Schwefel under multiplicative noise, where the confidence bands of all variants overlap. For $d \geq 20$, the variants with $\mathrm{G}>1$ generally match or outperform $\mathrm{G}=1$ on Ackley, Alpine-01 and Rotated Rastrigin. The gap is clearest on Ackley under additive noise, where, at $d=50$, the variants with $\mathrm{G}>1$ end below 18, as opposed to about 19 for $\mathrm{G}=1$. Schwefel is again the exception, as $\mathrm{G}=1$ performs best and performance generally degrades as $\mathrm{G}$ increases.

These results may be explained by the division of the fixed budget between the two phases. A larger batch spends more evaluations on exploring the domain and leaves fewer for the local phase. At $d=4$, the domain is small enough to be covered with few global evaluations, so additional global candidates mainly reduce the budget of the local phase. Schwefel has many deep local minima and no global shape that guides the search toward the optimum, so progress comes mainly from the local search around the incumbent. In higher dimension, covering the domain is harder, and the additional global candidates improve the results. Although $\mathrm{G}=1$ performs best at $d=4$, we keep $\mathrm{G}=3$ in all dimensions because high dimension is our main interest.
\begin{figure}
    \centering
    \includegraphics[width=0.9\linewidth]{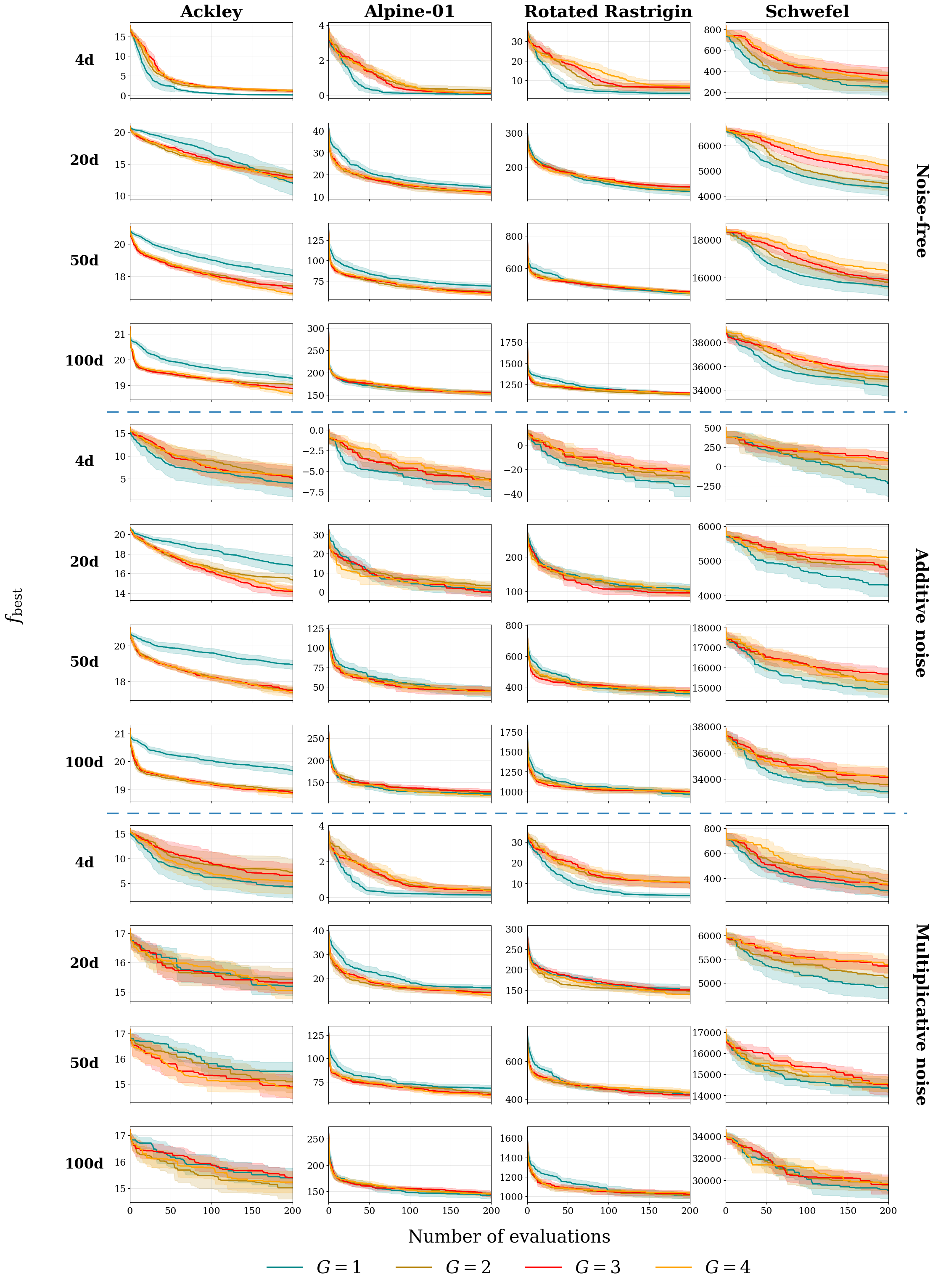}
    \caption{ Comparison of STREGO with global batch sizes $\mathrm{G} \in \{1, 2, 3, 4\}$, all other settings held fixed, across dimensions and noise settings, averaged over 20 trials.}
    \label{fig:batchsize_ablation}
\end{figure}

\end{document}